\documentclass{amsart}
\usepackage{amsfonts, amsbsy, amsmath, amssymb}

\usepackage{rotating}

\newtheorem{thm}{Theorem}[section]
\newtheorem{lem}[thm]{Lemma}
\newtheorem{cor}[thm]{Corollary}
\newtheorem{prop}[thm]{Proposition}
\newtheorem{exmp}[thm]{Example}

\newtheorem{algo}[thm]{Algorithm}
\newtheorem{rmk}[thm]{Remark}
\newtheorem{ques}[thm]{Question}
\numberwithin{equation}{section}

\theoremstyle{definition}
\newtheorem{defn}[thm]{Definition}

\input prepictex.tex
\input pictex.tex
\input postpictex.tex

\def\boxprec{
\text{
\beginpicture
\setcoordinatesystem units <2.5mm,2.5mm> point at 0 0
\plot 0 0  1 0  1 1  0 1  0 0 /
\put{$\scriptstyle \prec$} at 0.5 0.5
\endpicture
}}

\def\boxsucc{
\text{
\beginpicture
\setcoordinatesystem units <2.5mm,2.5mm> point at 0 0
\plot 0 0  1 0  1 1  0 1  0 0 /
\put{$\scriptstyle \succ$} at 0.5 0.5
\endpicture
}}

\begin{document}

\title[Polynomial Ring over a Commutative Chain Ring]{Lattice of Ideals of the Polynomial Ring over a Commutative Chain Ring}

\author[Xiang-dong Hou]{Xiang-dong Hou*}
\address{Department of Mathematics and Statistics,
University of South Florida, Tampa, FL 33620}
\email{xhou@usf.edu}
\thanks{* Research partially supported by NSA Grant H98230-12-1-0245.}



\begin{abstract}
Let $R$ be a commutative chain ring. We use a variation of Gr\"obner bases to study the lattice of ideals of $R[x]$. Let $I$ be a proper ideal of $R[x]$. We are interested in the following two questions: When is $R[x]/I$ Frobenius? When is $R[x]/I$ Frobenius and local? We develop algorithms for answering both questions. When the nilpotency of $\text{rad}\,R$ is small, the algorithms provide explicit answers to the questions.
\end{abstract}

\maketitle

\section{Introduction}

All rings are assumed to have identity.
A {\em commutative chain ring} is a commutative ring $R$ whose Jacobson radical $\text{rad}\,R$ is principally generated by a nilpotent element. A discrete valuation ring modulo a nontrivial ideal is an example of commutative chain ring. Let $R$ be a commutative chain ring and let $I$ be a proper ideal of $R[x]$. We are interested in the following two questions.

\begin{ques}\label{Q1.1}
When is $R[x]/I$ Frobenius? (The definition of a Frobenius ring is given in Section 2.) 
\end{ques}

\begin{ques}\label{Q1.2}
When is $R[x]/I$ Frobenius and local?  
\end{ques}
 
Our interest in these questions was inspired by applications of finite Frobenius rings in coding theory and combinatorics. A finite ring $R$ is said to have the {\em extension property} if for every two $R$-submodules ${}_RC_1,\ {}_RC_2\subset R^n$ and an $R$-isomorphism $g:C_1\to C_2$ which preserves the Hamming weight, $g$ can be extended to an automorphism of $R^n$ which preserves the Hamming weight. J. Wood \cite{Woo99,Woo08} proved that a finite ring has the extension property if and only if it is Frobenius. It is also known that the MacWilliams identity holds for linear codes over finite Frobenius rings \cite{Hon-Lan01, Woo99}. For more works on codes over finite Frobenius rings, see 
\cite{BGKS10,Din10, Din-Lop-Sza09,Dou-Kim-Kul-Liu10,Dou-Yil-Kar11, Gre-Nec-Wis04, Nec08, Nec-Kuz-Mar97,Nor-Sal03, Yil-Kar10} 
and the references therein. Finite Frobenius rings have also been used to construct combinatorial objects (bent functions, Hadamard matrices, and partial difference sets) that are defined in terms of the characters of finite abelian groups \cite{Hou00,Hou03,Hou-Nec07,McG-War09}. It is known that every commutative Frobenius ring is a direct product of finitely many commutative Frobenius local rings \cite[Theorem 15.27]{Lam98}. Examples of finite Frobenius rings in literature frequently appear in the form of $R[x]/I$, where $R$ is a finite commutative chain ring. Therefore, it is desirable to investigate this type of rings more systematically. In doing so, we will not insist on $R$ being finite since the requirement is unnecessary. 

Returning to Questions~\ref{Q1.1} and \ref{Q1.2}, we know that $R[x]/I$ is Frobenius if and only if the annihilator of every maximal ideal of $R[x]/I$ is a minimal ideal (Theorem~\ref{T2.3}) and that $R[x]/I$ is Frobenius and local if and only if it has a unique minimal ideal (Theorem~\ref{T2.4}). To make use of these characterizations, we need to answer certain related questions about the lattice of ideals of $R[x]$. What are the maximal ideals of $R[x]$ that contain $I$? How to compute the annihilators of maximal ideals in $R[x]/I$? How to determine the ideals that are immediate above $I$ in the ideal lattice of $R[x]$? We will develop effective algorithms for answering these questions. The algorithms are essentially based on a variation of Gr\"obner bases of ideals in $R[x]$. The complexity of the ideal lattice of $R[x]$ increases rapidly as the nilpotency of $\text{rad}\,R$ increases. In general, the answers to Questions~\ref{Q1.1} and \ref{Q1.2} provided in this paper are algorithmic but not explicit. However, when the nilpotency of $\text{rad}\,R$ is small, the algorithms do provide explicit answers.

In Section 2, we gather some basic facts about Frobenius rings and chain rings. Starting from Section 3, we focus on the polynomial ring $R[x]$ over a commutative chain ring $R$.  For each ideal of $R[x]$, we introduce the notion of {\em canonical sequence}, which contains the same information as the Gr\"obner basis of the ideal, and the notion of {\em invariant sequence}, which is a sequence of polynomials over the residue field $F=R/\text{rad}\,R$ that resemble the invariant factors of a finitely generated module over a PID. Roughly speaking, invariant sequences provide a rather coarse sketch of the ideal lattice of $R[x]$ and canonical sequences refine the picture. We also describe the Gr\"obner basis of any given ideal of $R[x]$ using the canonical sequence of the ideal. 
Since the polynomial ring is univariate, canonical sequences, invariant sequences, and Gr\"obner bases are all easy to compute.
In Section 4, we describe the maximal and minimal ideals of $R[x]/I$. Section 5 contains an example where all ideals of $R[x]$ with invariant sequence dividing a given one are enumerated and the relations among those ideals are determined. In Section 6, we outline the master algorithms for answering Questions~\ref{Q1.1} and \ref{Q1.2}, and we provide supporting algorithms for the specific tasks in the master algorithms.  
In Section 7, we consider commutative chain rings $R$ for which $\text{rad}\,R$ has small nilpotency. In such cases, the algorithms in Section 6 produce explicit answers to Questions~\ref{Q1.1} and \ref{Q1.2}.

All algorithms in the paper except Algorithms~\ref{A3.3} are not in the pseudo code format; they are normal text passages containing descriptions and justifications of computational steps, and they are labeled for easy reference.

\section{Frobenius Rings and Chain Rings} 

Let $R$ be a ring and ${}_RM$ a left $R$-module. The {\em socle} of ${}_RM$, denoted by $\text{soc}({}_RM)$, is the sum of all simple submodules of $M$. The socle of a right $R$-module is defined in the same way. A {\em Frobenius ring} is a left artinian ring $R$ satisfying $\text{soc}({}_RM)\cong{}_R\bar R$ and $\text{soc}(M_R)\cong\bar R_R$, where $\bar R=R/\text{rad}\,R$.

\begin{thm}\label{T2.1} \cite{Woo99} A finite ring $R$ is Frobenius if and only if the abelian group $(R,+)$ has a character $\chi$ such that $\ker\chi$ does not contain any nonzero left ideal of $R$. (The character $\chi$ is called a {\em left generating character} of $R$.) The statement is left-right symmetric.
\end{thm}

\begin{thm}\label{T2.2} \cite{Hou-Nec07} A finite ring $R$ is Frobenius and local if and only if it has a unique minimal left ideal. The statement is left-right symmetric. (If $R$ is a finite Frobenius local ring, its minimal left ideal coincides with the minimal right ideal.)
\end{thm}

\begin{thm}\label{T2.3} A commutative ring $R$ is Frobenius if and only if $R$ is artinian and for every maximal ideal $M$ of $R$, its annihilator $\text{\rm ann}(M)$ is a minimal ideal of $R$.
\end{thm}

\begin{proof}($\Rightarrow$) This obvious since $\text{ann}(\ )$ is an inclusion-reversing bijection from the set of all ideals of $R$ to itself \cite[Theorem~15.1]{Lam98}.

($\Leftarrow$) $1^\circ$ We claim that for every minimal ideal $I$ of $R$, $\text{ann}(I)$ is a maximal ideal of $R$ and $\text{ann}(\text{ann}(I))=I$.

Write $I=aR$, where $a\in R$. Then $R/\text{ann}(I)=R/\text{ann}(a)\cong aR$, which is simple. Thus $\text{ann}(I)$ is a maximal ideal of $R$. Since $I\subset \text{ann}(\text{ann}(I))$ and $\text{ann}(\text{ann}(I))$ is minimal, we have $I=\text{ann}(\text{ann}(I))$.

\medskip
$2^\circ$ We claim that for any maximal ideal $M$ of $R$, $\text{ann}(\text{ann}(M))=M$. 

For this claim, we only need the assumption that $R$ is artinian. (The assumption that $\text{ann}(M)$ is minimal is unnecessary.) We first show that $\text{ann}(M)\ne 0$. Since $R$ is artinian, the descending chain $M^1\supset M^2\supset\cdots$ must stabilize, say, $M^{n+1}=M^n$. Since $M^n$ is finitely generated ($R$ is noetherian) and $MM^n=M^n$, by Nakayama's lemma \cite[(1.M)]{Mat80}, there exists $x\in M$ such that $(1+x)M^n=0$. We may assume that $n$ is the smallest positive integer with this property. Then $0\ne(1+x)M^{n-1}\subset\text{ann}(M)$. Hence $\text{ann}(M)\ne 0$. Now we have $M\subset \text{ann}(\text{ann}(M))\subsetneq R$. It follows that $M= \text{ann}(\text{ann}(M))$.

\medskip
$3^\circ$ Since $R$ is artinian, we have $\text{soc}(R)=I_1\oplus \cdots\oplus I_k$, where each $I_i$ is a minimal ideal of $R$. For each maximal ideal $M$ of $R$, we have $\text{ann}(M)\subset I_1\oplus\cdots\oplus I_k$. Hence
\[
M=\text{ann}(\text{ann}(M))\supset\text{ann}(I_1\oplus \cdots\oplus I_k)=\text{ann}(I_1)\cap\cdots\cap\text{ann}(I_k).
\]
Thus $\text{rad}\, R=\text{ann}(I_1)\cap\cdots\cap\text{ann}(I_k)$. Write $I_i=a_iR$, $a_i\in R$. Define
\[
\begin{array}{cccc}
\phi:&R&\longrightarrow&a_1R\oplus\cdots\oplus a_kR\cr
&x&\longmapsto&a_1x+\cdots+a_kx
\end{array}
\]
Clearly, $\ker\phi=\text{ann}(I_1)\cap\cdots\cap\text{ann}(I_k)=\text{rad}\, R$. By $1^\circ$, $\text{ann}(I_1)\ne\text{ann}(I_i)$ for all $2\le i\le k$. Choose $b_i\in\text{ann}(I_i)\setminus\text{ann}(I_1)$, $2\le i\le k$. Then $b_2\cdots b_k\in\bigl(\text{ann}(I_2)\cap\cdots\cap\text{ann}(I_k)\bigr)\setminus\text{ann}(I_1)$. Hence $\phi(b_2\cdots b_k)=a_1b_2\cdots b_k\in I_1\setminus\{0\}$. It follows that $\phi$ is onto. Therefore
\[
R/\text{rad}\,R\cong a_1R\oplus\cdots\oplus a_kR=\text{soc}(R).
\]
So $R$ is Frobenius.
\end{proof}

\begin{thm}\label{T2.4}
A commutative ring $R$ is Frobenius and local if and $R$ is artinian and has a unique minimal ideal.
\end{thm}

\begin{proof}
($\Rightarrow$) This is obvious since $R$  has a unique maximal ideal and $\text{ann}(\ )$ is an inclusion-reversing bijection from the set of all ideals of $R$ to itself.

($\Leftarrow$) Let $aR$ ($a\in R$) be the unique minimal ideal of $R$. Let $M$ be any maximal ideal of $R$. By the proof of Theorem~\ref{T2.3}, ($\Leftarrow$), step $2^\circ$, we have $\text{ann}(M)\ne 0$. Then we must have $a\in\text{ann}(M)$, i.e., $M\subset \text{ann}(a)$. Therefore $\text{ann}(a)$ is the unique maximal ideal of $R$, and hence $R$ is local. Since $R/\text{rad}\,R=R/\text{ann}(a)\cong aR=\text{soc}(R)$, $R$ is Frobenius.
\end{proof} 

\begin{rmk}\label{R2.5}\rm
If a commutative ring $R$ has a unique minimal ideal, it may not be artinian, hence may not be Frobenius. For example, let $R=F[x]\times F$, where $F$ is any field. Then $0\times F$ is the unique minimal ideal of $R$, but $R$ is not artinian since $F[x]$ is not artinian.
\end{rmk}

We define a {\em chain ring} to be a ring whose left ideals form a {\em finite} chain. (We note that this definition differs from others in literature in that we require the ring to have only finitely many left ideals; for comparison, see for example \cite{Bru-Tor76}.) The definition is left-right symmetric. For a chain ring $R$, we have the following facts:
\begin{itemize}
  \item $R$ is local and there is a nilpotent element $\pi\in R$ such that $\text{rad}\,R=\pi R=R\pi$.
  \item All one-sided ideals of $R$ are two-sided.
  \item All ideals of $R$ form a chain $R\supset \pi R\supset\pi^2 R\supset\cdots\supset\pi^n R=\{0\}$, where $n$ is the nilpotency of $\pi$.
\end{itemize}  
It is known that \cite{Cla-Dra73} \cite[\S5.2]{Nec08} that every finite chain ring $R$ is of the form
\[
R=S[x,\sigma]/\langle g,p^{n-1}x^t\rangle,
\]
where  
\begin{itemize}
  \item $S$ is a {\em Galois ring} of characteristic $p^n$;
  \item $\sigma\in\text{Aut}(S)$;
  \item $S[x;\sigma]$ is the skew polynomial ring in $x$ over $S$ satisfying $xa=\sigma(a)x$ for all $a\in S$;
  \item $g\in S[x;\sigma]$ is an {\em Eisenstein polynomial} of degree $k$, i.e., $g=x^k-p(a_{k-1}x^{k-1}+\cdots+a_0)$, $a_i\in S$, $a_0\in S^\times$, where $S^\times$ is the group of units of $S$;
  \item $t=k$ when $n=1$, and $1\le t\le k$ when $n>1$.
\end{itemize}   
For the definition of Galois rings and the description of their automorphism groups, see \cite[\S4.3]{Nec08} \cite[Chapter 14]{Wan03}. In the above notation, every finite commutative chain ring is of the form 
\[
R=S[x]/\langle g,p^{n-1}x^t\rangle.
\]

\section{Gr\"obner Bases, Canonical Sequences, and Invariant Sequences} 

From now on, $R$ always denote a commutative chain ring. Write $\text{rad}\,R=\pi R$, where $\pi\in R$ is of nilpotency $n$. Let $F=R/\text{rad}\,R$ be the residue field of $R$. The symbol $\langle\ \rangle$ is reserved for ideals generated in $R[x]$ unless a different ambient ring is clearly stated in the context. The natural homomorphism from $R$ to $F$ and its induced homomorphism from $R[x[$ to $F[x]$ are both denoted by $\overline{(\ )}$.

\begin{prop}\label{P3.1} Let $I$ be an ideal of $R[x]$ and $f_0,\dots,f_{n-1}\in R[x]$. The following statements are equivalent.
\begin{itemize}
  \item [(i)] $I\cap\langle\pi^i\rangle=\langle\pi^if_i,\dots,\pi^{n-1}f_{n-1}\rangle$ for all $0\le i\le n-1$.
  \item [(ii)] For each $0\le i\le n-1$, $\pi^if_i\in I$ and 
\begin{equation}\label{3.1}
\deg\overline{f_i}=\min\{\deg\overline f:f\in R[x],\ \pi^if\in I\}.
\end{equation}
(We define $\deg 0=\infty$.)
\item[(iii)] $I=\langle\pi^0f_0,\dots,\pi^{n-1}f_{n-1}\rangle$ and for each $0\le i\le n-2$,
\begin{equation}\label{3.2}
\langle\pi^{i+1}f_{i+1},\dots,\pi^{n-1}f_{n-1}\rangle\ni
\begin{cases}
\pi^i f_i&\text{if}\ \overline{f_i}=0,\cr
\pi^{i+1}f_i&\text{if}\ \overline{f_i}\ne 0.
\end{cases}
\end{equation}
\end{itemize}
\end{prop}

\begin{proof}
(i) $\Rightarrow$ (ii). For each $f\in R[x]$ with $\pi^if\in I$, we have $\pi^if\in\langle \pi^if_i,\dots,\pi^{n-1}f_{n-1}\rangle$. Hence $\overline{f_i}\mid\overline f$, which gives $\deg\overline{f_i}\le \deg\overline f$. Therefore \eqref{3.1} holds.

\medskip
(ii) $\Rightarrow$ (i). We prove (i) by (backward) induction on $i$. When $i=n$, there is nothing to prove. Assume $0\le i\le n-1$. It suffices to show that $I\cap\langle\pi^i\rangle\subset\langle\pi^if_i,\dots,\pi^{n-1}f_{n-1}\rangle$. For each $g\in I\cap\langle\pi^i\rangle$, write $g=\pi^if$, where $f\in R[x]$. By \eqref{3.1}, we have $\overline{f_i}\mid\overline f$. So $f=hf_i+\pi r$ for some $h,r\in R[x]$. It follows that $g=\pi^if=h\pi^if_i+\pi^{i+1}r$, where $\pi^{i+1}r\in I\cap\langle\pi^{i+1}\rangle$. Hence $g\in \langle\pi^if_i\rangle+I\cap\langle\pi^{i+1}\rangle= \langle\pi^if_i,\dots,\pi^{n-1}f_{n-1}\rangle$ by the induction hypothesis.

\medskip
(i) $\Rightarrow$ (iii). Obvious.

\medskip
(iii) $\Rightarrow$ (i). We prove (i) by (forward) induction on $i$. When $i=0$, (i) is given in (iii). Assume $0<i\le n-2$. It suffices to show that $I\cap\langle \pi^i\rangle\subset\langle\pi^if_i,\dots,\pi^{n-1}f_{n-1}\rangle$. For each $g\in I\cap\langle \pi^i\rangle$, we have $g\in\langle \pi^{i-1}\rangle=\langle\pi^{i-1}f_{i-1},\dots,\pi^{n-1}f_{n-1}\rangle$ by the induction hypothesis. If $\overline{f_{i-1}}=0$, then $\pi^{i-1}f_{i-1}\in \langle\pi^if_i,\dots,\pi^{n-1}f_{n-1}\rangle$, and hence $g\in \langle\pi^if_i,\dots,\pi^{n-1}f_{n-1}\rangle$. Assume $\overline{f_{i-1}}\ne 0$. We have $g=a_{i-1}\pi^{i-1}f_{i-1}+\cdots+a_{n-1}\pi^{n-1}f_{n-1}$ for some $a_{i-1},\dots,a_{n-1}\in R[x]$. Since $g\in\langle\pi^i\rangle$, we must have $a_{i-1}\in\langle \pi\rangle$, say, $a_{i-1}=a_{i-1}'\pi$ for some $a_{i-1}'\in R[x]$. Then $g=a_{i-1}'\pi^if_{i-1}+a_i\pi^if_i+\cdots+a_{n-1}\pi^{n-1}f_{n-1}\in\langle\pi^if_i,\dots,\pi^{n-1}f_{n-1}\rangle$ since $\pi^if_{i-1}\in\langle\pi^if_i,\dots,\pi^{n-1}f_{n-1}\rangle$.
\end{proof}   

\begin{defn}\label{D3.2}
A sequence of polynomials $(f_0,\dots,f_{n-1})$ in $R[x]$ satisfying the conditions in Theorem~\ref{3.1} is called a {\em canonical sequence} of the ideal $I$. The sequence $(\overline{f_0},\dots,\overline{f_{n-1}})$ is called the {\em invariant sequence} of $I$. Clearly, $\overline{f_{n-1}}\mid\overline{f_{n-2}}\mid\cdots\mid\overline{f_0}$.
\end{defn} 

The canonical sequence of an ideal $I$ is not unique. The invariant sequence of $I$ is unique if we require each polynomial in the sequence to be monic or $0$. The canonical sequence of $I$ can be easily obtained from any finite set of generators of $I$ by the following algorithm.

\begin{algo}\label{A3.3}\rm (Canonical sequence)

\noindent{\bf Input:} $G:=\{g_1,\dots,g_m\}\subset R[x]$

\noindent{\bf Output:} A canonical sequence $C$ of $\langle g_1,\dots,g_m\rangle$

\noindent
\begin{tabbing} {\bf Initialization:} \= $C:=\emptyset$;\\
\> $i:=0$;
\end{tabbing}

\noindent{\bf While} $i<n$ {\bf do}
\begin{list}{}{\leftmargin 5mm  \itemsep 2pt}
  \item [] Write $\text{gcd}\{\overline g:g\in G\}=\sum_{g\in G}\alpha_g\,\overline g$, $\alpha_g\in F[x]$;
  \item [] For each $g\in G$, lift $\alpha_g$ to $\widetilde{\alpha_g}\in R[x]$;
  \item [] Let $f_i\equiv\sum_{g\in G}\widetilde{\alpha_g}\,g\pmod{\pi^{n-i}}$;
  \item [] $C:=C\cup\{f_i\}$;
  \item [] For each $g\in G$, find $g'\in R[x] \pmod{\pi^{n-i-1}}$ such that $g\equiv \pi g'\pmod{\langle f_i,\pi^{n-i}\rangle}$, in $G$ replace $g$ by $g'$;
  \item [] $G:=G\cup\{f_i\}$;
  \item [] $i:=i+1$.
\end{list}  
\end{algo}

One can easily see that the output $\{f_0,\dots,f_{n-1}\}$ in Algorithm~\ref{A3.3} satisfies \eqref{3.1}, which implies the correctness of the algorithm.

Let $\frak F$ be a set of coset representatives of $\text{rad}\,R$ in $R$ such that $0\in\frak F$ and let $\frak F[x]\subset F[x]$ denote the set of polynomials with coefficients in $\frak F$. Let $(f_0,\dots,f_{n-1})$ be a canonical sequence of an ideal $I$. For each $f\in R[x]$, there exist $a_0,r_1\in R[x]$, $b_0\in\frak F[x]$ such that 
\[
f=a_0f_0+b_0+\pi r_1,
\]
where $b_0=0$ or $\deg b_0<\deg \overline{f_0}$. Next, there exist $a_1,r_2\in R[x]$, $b_1\in\frak F[x]$ such that
\[
\pi r_1=\pi a_1f_1+\pi b_1+\pi^2r_2,
\]
where $b_1=0$ or $\deg b_1<\deg\overline{f_1}$. Continuing this way, we have 
\begin{equation}\label{3.3}
f=a_0f_0+a_1\pi f_1+\cdots+a_{n-1}\pi^{n-1}f_{n-1}+b_0+\pi b_1+\cdots+\pi^{n-1}b_{n-1},
\end{equation}
where $a_i\in R[x]$, $b_i\in\frak F[x]$, $b_i=0$ or $\deg b_i<\deg\overline{f_i}$. It is easy to see that $b_0,\dots,b_{n-1}$ are unique. (If we require $a_i\in\frak F[x]$, they are also unique.) We call $b_0+\pi b_1+\cdots+\pi^{n-1}b_{n-1}$ the {\em reduction} of $f$ by the canonical sequence $(f_0,\dots,f_{n-1})$. It is clear that $f\in I$ if and only if the reduction of $f$ by $(f_0,\dots,f_{n-1})$ is $0$. Note that $\{\pi^0f_0,\dots,\pi^{n-1}f_{n-1}\}$ is not necessarily a Gr\"obner basis of $I$ since the reduction in \eqref{3.3} is not solely based on the term order, but rather on a combination of the $\pi$-adic order and the term order (with priority given to the $\pi$-adic order). Let $f_i'$ be the reduction of $f_i$ by $(0,\dots,f_{i+1},\dots,f_{n-1})$. Then $(f_0',\dots,f_{n-1}')$ is also a canonical sequence of $I$ and has the additional property that $\deg f_i'=\deg\overline{f_i'}$.

We recall the notion of strong Gr\"obner bases from \cite[Definition~4.5.6]{Ada-Lou94}. The leading term and the leading coefficient of a polynomial in $R[x]$ are denoted by $\text{lt}(\ )$ and $\text{lc}(\ )$, respectively. Let $G=\{g_1,\dots,g_t\}\subset R[x]$. We say that $G$ is a {\em strong Gr\"obner basis} of the ideal $I=\langle g_1,\dots,g_t\rangle$ if for every $f\in I$, there exists $i\in\{1,\dots,t\}$ such that $\text{lt}(g_i)\mid\text{lt}(f)$. Strong Gr\"obner bases are Gr\"obner bases \cite[Lemma~4.5.8]{Ada-Lou94}.

\begin{prop}\label{P3.4}
Let $(f_0,\dots,f_{n-1})$ be a canonical sequence of $I$ such that $\deg f_i=\deg\overline{f_i}$, $0\le i\le n-1$. Then $G=\{\pi^0f_0,\dots,\pi^{n-1}f_{n-1}\}$ is a strong Gr\"obner basis of $I$.
\end{prop}

\begin{proof}
Let $f\in I$. By \eqref{3.3}, we may write
\begin{equation}\label{3.4}
f=a_0\pi^0f_0+\cdots+a_{n-1}\pi^{n-1}f_{n-1},
\end{equation}
where $a_i\in\frak F[x]$. Since $\deg\,f_i=\deg\,\overline{f_i}$, the $\pi$-adic order of $\text{lc}(a_i\pi^if_i)$ is $i$ when $a_i\ne 0$. Therefore at the right side of \eqref{3.4}, the leading terms of the summands never cancel. Hence $\text{lt}(f)=\text{lt}(a_i\pi^if_i)=a_i\text{lt}(\pi^if_i)$ for some $i\in\{0,\dots,n-1\}$. Therefore $G$ is a strong Gr\"obner basis of $I$.
\end{proof}    

The result of Proposition~\ref{P3.4} is not new. See \cite[Theorem~3.2]{Nor-Sal03} for a characterization of minimal strong Gr\"obner bases in $R[x]$.


\section{Maximal and Minimal Ideals of $R[x]/I$}

\begin{prop}\label{P4.1}
Let $I$ be a proper ideal of $R[x]$ with invariant sequence $(\alpha_0,\dots,\alpha_{n-1})$. Then the maximal ideals of $R[x]/I$ are precisely $\langle r,\pi\rangle/I$, where $r\in R[x]$ is such that $\overline r$ is an irreducible factor of $\alpha_0$.
\end{prop}

\begin{proof}
The maximal ideals of $R[x]$ are $\langle r,\pi\rangle$, where $\overline r\in F[x]$ is irreducible. The maximal ideals containing $I$ are those with $\overline r\mid\alpha_0$.
\end{proof}

To describe the minimal ideals of $R[x]/I$, we first introduce a partial order among invariant sequences. Let
\[
\mathcal A=\{(\alpha_0,\dots,\alpha_{n-1})\in F[x]^n:\alpha_{n-1}\mid\alpha_{n-2}\mid\cdots\mid\alpha_0\}.
\]
If $(\alpha_0,\dots,\alpha_{n-1})\in\mathcal A$ and $c_0,\dots,c_{n-1}\in F^\times$, $(\alpha_0,\dots,\alpha_{n-1})$ and $(c_0\alpha_0,\dots, c_{n-1}\alpha_{n-1})$ are considered the same element in $\mathcal A$. For $A=(\alpha_0,\dots,\alpha_{n-1})$ and $B=(\beta_0,\dots,\beta_{n-1})$ in $\mathcal A$, define $A\prec B$ if $\alpha_i\mid \beta_i$ for all $0\le i\le n-1$. If $B$ is immediately above $A$, i.e., $B=(\alpha_0,\dots,\alpha_{i-1},\gamma\alpha_i,\alpha_{i+1},\dots,\alpha_{n-1})$ for some $i$ with $\alpha_i\ne0$ and some irreducible factor $\gamma$ of $\alpha_{i-1}/\alpha_i$, we write $A\boxprec\, B$. If $I$ and $J$ are ideals of $R[x]$ with invariant sequences $A$ and $B$, respectively, then $I\subset J$ implies $A\succ B$.

\begin{lem}\label{L4.2} Let $I$ and $J$ be two ideals of $R[x]$. If $I\subset J$ and $I,J$ have the same invariant sequence, then $I=J$.
\end{lem}

\begin{proof}
Let $(f_0,\dots,f_{n-1})$ be a canonical sequence of $I$ and $(g_0,\dots,g_{n-1})$ a canonical sequence of $J$. We show by induction that $I\cap\langle\pi^i\rangle=J\cap\langle\pi^i\rangle$ for all $0\le i\le n$. The case $i=n$ requires no proof. Assume $0\le i\le n-1$. Since $\overline{f_i}=\overline{g_i}$, we have $\pi^i f_i=\pi^ig_i+\pi^{i+1}r$ for some $r\in R[x]$. Then $\pi^{i+1}r\in J\cap\langle\pi^{i+1}\rangle=I\cap\langle\pi^{i+1}\rangle$ by the induction hypothesis. Thus $\pi^ig_i=\pi^if_i+\pi^{i+1}r\in I\cap\langle\pi^i\rangle$. It follows that $J\cap\langle\pi^i\rangle=\langle\pi^ig_i\rangle+J\cap\langle\pi^{i+1}\rangle\subset I\cap\langle\pi^i\rangle$.
\end{proof}

\begin{prop}\label{P4.3}
Let $I$ and $J$ be ideals of $R[x]$ such that $I\subset J$ and $I\ne R[x]$. Let $(f_0,\dots,f_{n-1})$ and $(g_0,\dots,g_{n-1})$ be canonical sequences of $I$ and $J$, respectively. Then $J/I$ is a minimal ideal of $R[x]/I$ if and only if $(\overline{g_0},\dots,\overline{g_{n-1}})\boxprec\,(\overline{f_0},\dots,\overline{f_{n-1}})$.
\end{prop}

\begin{proof}
($\Leftarrow$) Let $K$ be an ideal of $R[x]$ such that $I\subset K\subset J$. Let $(h_0,\dots,h_{n-1})$ be a canonical sequence of $K$. Then $(\overline{f_0},\dots,\overline{f_{n-1}})\succ(\overline{h_0},\dots,\overline{h_{n-1}})\succ(\overline{g_0},\dots,\overline{g_{n-1}})$. By Lemma~\ref{L4.2}, $K=I$ or $J$.

($\Rightarrow$) Let $i$ be the smallest integer such that $\overline{f_i}\ne\overline{g_i}$. Let $r\in R[x]$ be such that $\overline r$ is an irreducible factor of $\overline{f_i}/\overline{g_i}$. Let $K=I+\langle\pi^i rg_i\rangle+J\cap\langle\pi^{i+1}\rangle$. Then $I\subset K\subset J$. By Proposition~\ref{P3.1} (ii), $K$ has a canonical sequence of the form $(f_0,\dots,f_{i-1},rg_i,h_{i+1},\dots,h_{n-1})$. Since $\overline r\,\overline{g_i}\ne\overline{g_i}$, we have $K\ne J$. Since $J/I$ is minimal, we must have $K=I$. Since $J\cap\langle\pi^{i+1}\rangle\subset K\subset J$,  we have $J\cap\langle\pi^{i+1}\rangle=K\cap\langle\pi^{i+1}\rangle$. Therefore
\[
\begin{split}
&(\overline{f_0},\dots,\overline{f_{n-1}})\cr
=\,&(\overline{f_0},\dots,\overline{f_{i-1}},\overline r \,\overline{g_i},\overline{h_{i+1}},\dots,\overline{h_{n-1}})\kern 1cm (\text{since}\ I=K)\cr
=\,&(\overline{g_0},\dots,\overline{g_{i-1}},\overline r\,\overline{g_i},\overline{h_{i+1}},\dots,\overline{h_{n-1}})\kern 1cm (\text{since}\ \overline{f_j}=\overline{g_j}\ \text{for}\ 0\le j<i)\cr
=\,&(\overline{g_0},\dots,\overline{g_{i-1}},\overline r\,\overline{g_i},\overline{g_{i+1}},\dots,\overline{g_{n-1}})\kern 1.1cm (\text{since}\ J\cap\langle\pi^{i+1}\rangle=K\cap\langle\pi^{i+1}\rangle).
\end{split}
\]
Hence $(\overline{f_0},\dots,\overline{f_{n-1}})\boxsucc\;(\overline{g_0},\dots,\overline{g_{n-1}})$.
\end{proof}

\begin{rmk}\label{R4.4}\rm \
\begin{itemize}
  \item [(i)] Let $I,J$ be ideals of $R[x]$ with invariant sequences $A$ and $B$, respectively. If $I\subset J$, then $A\succ B$. The converse is false.
  \item [(ii)] Let $A,B\in\mathcal A$ with $A\succ B$. Given an ideal $I$ with invariant sequence $A$, there may be no ideal $J$ with invariant sequence $B$ such that $I\subset J$. Given an ideal $J$ with invariant sequence $B$, there may be no ideal $I$ with invariant sequence $A$ such that $I\subset J$. See Example~\ref{E5.2}.
  \item [(iii)] If $I\subset J$ are ideals of $R[x]$ with invariant sequences $A$ and $B$, respectively, by Proposition~\ref{P4.3}, there exist ideals $I_1,\dots,I_k$ with invariant sequences $A_1,\dots,A_k$, respectively, such that $A\boxsucc\, A_1\boxsucc\,\cdots\boxsucc\, A_k\boxsucc\, B$ and $I=I_1\subset \cdots\subset I_k\subset J$.
\end{itemize}  
\end{rmk}

\begin{thm}\label{T4.5}
Let $I$ be an ideal of $R[x]$ with canonical sequence $(f_0,\dots,f_{n-1})$, where $\overline{f_i}\ne 0$ and $\overline{f_i}\ne\overline{f_{i+1}}$ for some $i\in\{0,\dots,n-1\}$. Let $r\in R[x]$ be such that $\overline r$ is an irreducible factor of $\overline{f_i}/\overline{f_{i+1}}$. Then the ideals containing $I$ with invariant sequence $(\overline{f_0},\dots,\overline{f_{i-1}},\overline{f_i}/\overline r,\overline{f_{i+1}},\dots,\overline{f_{n-1}})$ are precisely of the form
\begin{equation}\label{4.1}
J=I+\langle\pi^ig\rangle,
\end{equation} 
where $g\in R[x]$ satisfies
\begin{itemize}
  \item [(i)] $\overline g=\overline{f_i}/\overline r$;
  \item [(ii)] $\pi^{i+1}g\in I$;
  \item [(iii)] $\pi^irg\in I$.
\end{itemize}  
If $g,g'\in R[x]$ both satisfy (i) -- (iii), then $I+\langle\pi^ig\rangle=I+\langle\pi^ig'\rangle$ if and only if $\pi^i(g-g')\in I$.
\end{thm}

\begin{proof}
$1^\circ$ We first show that if $J$ is an ideal with invariant sequence $(\overline{f_0},\dots,\overline{f_{i-1}},$ $\overline{f_i}/\overline r,\overline{f_{i+1}},\dots,\overline{f_{n-1}})$ such that $J\supset I$, then $J$ is of the form \eqref{4.1}. Let $(g_0,\dots,g_{n-1})$ be a canonical sequence of $I$ and let $g=g_i$. Then $\overline g=\overline{f_i}/\overline r$. It follows that $\pi^ig\in J\setminus I$. Since $J/I$ is minimal (Proposition~\ref{P4.3}), we must have $J=I+\langle\pi^ig\rangle$. The ideals $J\cap\langle\pi^{i+1}\rangle$ and $I\cap\langle\pi^{i+1}\rangle$ have the same invariant sequence $(0,\dots,0,\overline{f_{i+1}},\dots,\overline{f_{n-1}})$. By Lemma~\ref{L4.2}, $J\cap\langle\pi^{i+1}\rangle=I\cap\langle\pi^{i+1}\rangle$. Clearly, both $\pi^{i+1}g$ and $\pi^i(f_i-rg)$ belong to $J\cap\langle\pi^{i+1}\rangle$. Since $J\cap\langle\pi^{i+1}\rangle=I\cap\langle\pi^{i+1}\rangle$, we have $\pi^{i+1}g, \pi^i(f_i-rg)\in I$, hence (ii) and (iii).

\medskip
$2^\circ$ Now assume that $J$ is of the form \eqref{4.1}. We show that the invariant sequence of $J$ is $(\overline{f_0},\dots,\overline{f_{i-1}},\overline g,\overline{f_{i+1}},\dots,\overline{f_{n-1}})$.

We have
\begin{equation}\label{4.2}
J\cap\langle\pi^i\rangle=\langle\pi^ig\rangle+I\cap\langle\pi^i\rangle.
\end{equation}
We will show that $J\cap\langle\pi^{i+1}\rangle=I\cap\langle\pi^{i+1}\rangle$, which implies that
\begin{equation}\label{4.3}
J\cap\langle\pi^j\rangle=I\cap\langle\pi^j\rangle,\qquad i+1\le j\le n-1.
\end{equation}
By \eqref{4.1} -- \eqref{4.3}, we see that the polynomials $f_0,\dots,f_{i-1},g,f_{i+1},\dots,f_{n-1}$ satisfy \eqref{3.1} for the ideal $J$. Hence $(f_0,\dots,f_{i-1},g,f_{i+1},\dots,f_{n-1})$ is a canonical sequence of $J$ and $(\overline{f_0},\dots,\overline{f_{i-1}},\overline g,\overline{f_{i+1}},\dots,\overline{f_{n-1}})$ is the invariant sequence of $J$.

It remains to show that $J\cap\langle\pi^{i+1}\rangle=I\cap\langle\pi^{i+1}\rangle$. Let $f\in J\cap\langle\pi^{i+1}\rangle$. Since $f\in J\cap\langle\pi^i\rangle=\langle\pi^ig\rangle+I\cap\langle\pi^i\rangle=\langle\pi^ig,\pi^if_i\rangle+I\cap\langle\pi^{i+1}\rangle$, we have $f=a\pi^ig+b\pi^if_i+u$, where $a,b\in R[x]$ and $u\in I\cap\langle\pi^{i+1}\rangle$. Since $f\in\langle\pi^{i+1}\rangle$, we have $\overline a\,\overline g+\overline b\,\overline{f_i}=0$, i.e., $\overline a (\overline{f_i}/{\overline r})+\overline b\,\overline{f_i}=0$. It follows that $\overline a+\overline b\,\overline r=0$, i.e., $a=-br+\pi c$ for some $c\in R[x]$. Thus
\[
f=\pi^i(ag+bf_i)+u=b\pi^i(-rg+f_i)+c\pi^{i+1}g+u\in  I\cap\langle\pi^{i+1}\rangle
\]
by (ii) and (iii). 

\medskip
$3^\circ$ Assume that $g,g'\in R[x]$ both satisfy (i) -- (iii) and let $J=I+\langle \pi^i g\rangle$, $J'=I+\langle\pi^i g'\rangle$. By Lemma~\ref{4.2}, $J=J'$ if and only if $J\subset J'$. It is proved in $2^\circ$ that $(f_0,\dots,f_{i-1},g',f_{i+1},\dots,f_{n-1})$ is a canonical sequence of $J'$. Thus 
\[
\begin{split}
J\subset J'\;&\Leftrightarrow \pi^ig\in\langle\pi^i g',\pi^{i+1}f_{i+1},\dots,\pi^{n-1}f_{n-1}\rangle\cr
 &\Leftrightarrow \pi^i(g-g')\in \langle\pi^{i+1}f_{i+1},\dots,\pi^{n-1}f_{n-1}\rangle\kern0.8cm (\text{since}\ \pi^i(g-g')\in\langle\pi^{i+1}\rangle)\cr
&\Leftrightarrow \pi^i(g-g')\in I.
\end{split}
\]
\end{proof}

For an ideal $I$ of $R[x]$ and an element $p\in R[x]$, we define
\[
[I:p]=\{a\in R[x]:ap\in I\}.
\]

\begin{thm}\label{T4.6}
In the notation of Theorem~\ref{T4.5}, choose $u\in R[x]$ such that $\overline u=\overline{f_i}/(\overline r\overline{f_{i+1}})$.  Then there exists an ideal containing $I$ with invariant sequence \break
$(\overline{f_0},\dots,\overline{f_{i-1}},\overline{f_i}/\overline r,\overline{f_{i+1}},\dots,\overline{f_{n-1}})$ if and only if $\pi^i(f_i-urf_{i+1})\in I\cap\langle\pi^{i+1}\rangle+r[I\cap\langle\pi^{i+2}\rangle:\pi]$.
\end{thm}

\begin{proof}
Let $g=uf_{i+1}+\pi h$, where $h\in R[x]$. Then $\overline g=\overline{f_i}/\overline r$. We have 
\[
\begin{split}
\pi^{i+1}g\,&=\pi^{i+1}uf_{i+1}+\pi^{i+2}h,\cr
\pi^irg\,&=\pi^iruf_{i+1}+\pi^{i+1}rh=\pi^if_i-v+\pi^{i+1}rh,
\end{split}
\]
where $v=\pi^i(f_i-urf_{i+1})\in \langle\pi^{i+1}\rangle$.
Therefore $g$ satisfies (ii) and (iii) of Theorem~\ref{T4.5} if and only if 
\begin{equation}\label{4.4}
\begin{cases}
\pi^{i+2}h\in I,\cr
v-\pi^{i+1}rh\in I.
\end{cases}
\end{equation}
By Theorem~\ref{T4.5}, it suffices to show that there exists $h\in R[x]$ satisfying \eqref{4.4} if and only if $v\in I\cap\langle\pi^{i+1}\rangle+r[I\cap\langle\pi^{i+2}\rangle:\pi]$. If $h$ satisfies \eqref{4.4}. Then $\pi^{i+1}h\in [I\cap\langle\pi^{i+2}\rangle:\pi]$ and 
\[
v\in r\pi^{i+1}h+I\cap\langle\pi^{i+1}\rangle\subset r[I\cap\langle\pi^{i+2}\rangle:\pi]+ I\cap\langle\pi^{i+1}\rangle.
\]
If $v\in r[I\cap\langle\pi^{i+2}\rangle:\pi]+ I\cap\langle\pi^{i+1}\rangle$, then $v\in rl+I\cap\langle\pi^{i+1}\rangle$ for some $l\in [I\cap\langle\pi^{i+2}\rangle:\pi]$. Write $l=\pi^{i+1}h$, where $h\in R[x]$. Then $h$ satisfies \eqref{4.4}.
\end{proof}

\begin{rmk}\label{R4.7}\rm
In Theorem~\ref{T4.6}, if $\pi^i(f_i-urf_{i+1})\in r[I\cap\langle\pi^{i+2}\rangle:\pi]+ I\cap\langle\pi^{i+1}\rangle$, write $\pi^i(f_i-urf_{i+1})\equiv rl\pmod{I\cap\langle\pi^{i+1}\rangle}$, where $l\in [I\cap\langle\pi^{i+2}\rangle:\pi]$. Let $\mathcal C$ be a set of coset representatives of $I\cap\langle\pi^{i+1}\rangle$ in $[I\cap\langle\pi^{i+2}\rangle:\pi]\cap [I\cap\langle\pi^{i+1}\rangle:r]$. Then 
\[
I+\langle \pi^iuf_{i+1}+l+c\rangle,\qquad c\in\mathcal C,
\]
is an enumeration of all distinct ideals containing $I$ with invariant sequence $(\overline{f_0},\dots,$ $\overline{f_{i-1}},\overline{f_i}/\overline r,\overline{f_{i+1}},\dots,\overline{f_{n-1}})$.
\end{rmk}


\section{An Example of Ideal Lattice}

Recall that $\frak F$ is a set of coset representatives of $\pi R$ in $R$ with $0\in\frak F$ and $\frak F[x]$ is the set of polynomials in $F[x]$ with coefficients in $\frak F$. Let $A=(\alpha_0,\dots,\alpha_{n-1})\in\mathcal A$ and let $\mathcal I^A$ be the set of all ideals of $R[x]$ with invariant sequence $A$. Assume that $\alpha_{i-1}=0$ but $\alpha_i\ne 0$. Choose $a_j\in\frak F[x]$ such that $\overline{a_j}=\alpha_j/\alpha_{j+1}$, $i\le j\le n-1$. (We define $\alpha_n=1$.) Let
\[
\mathcal B=\Bigl\{(b_{jk})_{i\le j\le n-2,\, 1\le k\le n-j-1}:b_{jk}\in\frak F[x],\ b_{jk}=0\ \text{or}\ \deg b_{jk}<\deg\frac{\alpha_{j+k}}{\alpha_{j+k+1}}\Bigr\}.
\]
For each $(b_{jk})\in\mathcal B$, construct inductively
\begin{equation}\label{5.1}
\begin{cases}
f_n=1,\cr
f_j=a_jf_{j+1}+\pi b_{j1}f_{j+2}+\cdots+\pi^{n-j-1}b_{j,n-j-1}f_n,&i\le j\le n-1,\cr
f_j=0,& 0\le j<i.
\end{cases}
\end{equation}
Then clearly $(f_0,\dots,f_{n-1})$ satisfies \eqref{3.2} and hence is the canonical sequence of the ideal $\langle\pi^0f_0,\cdots,\pi^{n-1}f_{n-1}\rangle$ with invariant sequence $A$.

\begin{prop}\label{P5.1}
In the above notation, the mapping
\begin{equation}\label{5.2}
\begin{array}{ccc}
\mathcal B&\longrightarrow&\mathcal I^A\cr
(b_{jk})&\longmapsto& \langle\pi^0f_0,\cdots,\pi^{n-1}f_{n-1}\rangle
\end{array}
\end{equation}
is a bijection.
\end{prop}

\begin{proof}
$1^\circ$ Let $I\in\mathcal I^A$. We show that $I$ has a canonical sequence of the form \ref{5.1}.

Let $(g_0,\dots,g_{n-1})$ be a canonical sequence of $I$ and set $g_n=1$. Since $\alpha_{i-1}=0$, we have $I\subset\langle\pi^i\rangle$. Thus we may assume $g_0=\cdots=g_{i-1}=0$. We show by induction that $g_n,\dots,g_i$ can be replaced with some $f_n,\dots,f_i$ of the form \eqref{5.1} without changing the ideal $I$. Assume that $g_n,\dots,g_{j+1}$ have been replaced by $f_n,\dots,f_{j+1}$, where $i\le j\le n-1$. Since $\overline{g_j}=\overline {a_j}\overline{f_{j+1}}$, we have $\pi^{j+1}g_j\equiv\pi^{j+1}a_jf_{j+1}\pmod{I\cap\langle\pi^{j+2}\rangle}$. Since $I\cap\langle\pi^{j+2}\rangle=\langle \pi^{j+2}f_{j+2},\dots,\pi^{n-1}f_{n-1}\rangle$, we have 
\begin{equation}\label{5.3}
\pi^{j+1}g_j\equiv a_j\pi^{j+1}f_{j+1}+c_{j1}\pi^{j+2}f_{j+2}\pmod{I\cap\langle\pi^{j+3}\rangle}
\end{equation}
for some $c_{j1}\in R[x]$. Let $b_{j1}\in\frak F[x]$ be such that $\overline{c_{j1}}\equiv \overline{b_{j1}}\pmod{\overline{f_{j+1}}/\overline{f_{j+2}}}$ and $b_{j1}=0$ or $\deg \overline{b_{j1}}<\deg(\overline{f_{j+1}}/\overline{f_{j+2}})$. Then 
\begin{equation}\label{5.4}
c_{j1}f_{j+2}\equiv b_{j1}f_{j+2}+d_1f_{j+1}\pmod{\langle\pi\rangle}
\end{equation}
for some $d_1\in R[x]$. By \eqref{5.3} and \eqref{5.4},
\[
\pi^{j+1}g_j\equiv a_j\pi^{j+1}f_{j+1}+b_{j1}\pi^{j+2}f_{j+2}+d_1\pi^{j+2}f_{j+1}\pmod{I\cap\langle\pi^{j+3}\rangle}.
\]
Continuing this way, we have
\[
\begin{split}
\pi^{j+1}g_j=\,&a_j\pi^{j+1}f_{j+1}+b_{j1}\pi^{j+2}f_{j+2}+\cdots+b_{j,n-j-2}\pi^{n-1}f_{n-1}\cr
&+d_1\pi^{j+2}f_{j+1}+\cdots+d_{n-j-2}\pi^{n-1}f_{n-2},
\end{split}
\]
where $b_{jk}\in\frak F[x]$, $b_{jk}=0$ or $\deg\overline {b_{jk}}<\deg(\overline{f_{j+k}}/\overline{f_{j+k+1}})$, $1\le k\le n-j-2$, $d_1,\dots,d_{n-j-2}\in R[x]$.
Therefore
\begin{equation}\label{5.5}
\pi^jg_j\equiv a_j\pi^{j}f_{j+1}+b_{j1}\pi^{j+1}f_{j+2}+\cdots+b_{j,n-j-2}\pi^{n-2}f_{n-1}+\pi^{n-1}c\pmod I
\end{equation}
for some $c\in R[x]$. Let $b_{j,n-j-1}\in\frak F[x]$ be such that $\overline c\equiv\overline{b_{j,n-j-1}}\pmod{\overline{f_{n-1}}}$ and $b_{j,n-j-1}=0$ or $\deg\overline{b_{j,n-j-1}}<\deg\overline{f_{n-1}}$. Then
\begin{equation}\label{5.6}
c\equiv b_{j,n-j-1}+d_{n-j-1}f_{n-1}\pmod{\langle\pi\rangle}
\end{equation}
for some $d_{n-j-1}\in R[x]$. By \eqref{5.5} and \eqref{5.6} we have
\[
\begin{split}
\pi^jg_j\,&\equiv a_j\pi^jf_{j+1}+b_{j1}\pi^{j+1}f_{j+2}+\cdots+b_{j,n-j-2}\pi^{n-2}f_{n-1}+b_{j,n-j-1}\pi^{n-1}f_n\pmod I\cr
&=\pi^j(a_jf_{j+1}+\pi b_{j1}f_{j+2}+\cdots+\pi^{n-j-1}b_{j,n-j-1}f_n).
\end{split}
\]
Hence we can replace $g_j$ with $a_jf_{j+1}+\pi b_{j1}f_{j+2}+\cdots+\pi^{n-j-1}b_{j,n-j-1}f_n$. 

\medskip
$2^\circ$ Now let $(b_{jk}),(b_{jk}')\in \mathcal B$ be different. We show that their corresponding ideals in \eqref{5.2} are different. Let $(f_0,\dots,f_{n-1})$ and  $(f_0',\dots,f_{n-1}')$
be canonical sequences corresponding to $(b_{jk})$ and $(b_{jk}')$, respectively. Let $u,v$ ($i\le u\le n-2,\ 1\le v\le n-u-1$) be integers such that $b_{uv}\ne b_{uv}'$ but $b_{jk}=b_{jk}'$ for all $j,k$ with $j>u$ or $j=u$ and $k<v$. The by \eqref{5.1}, $f_j=f_j'$ for all $j>u$ and 
\begin{equation}\label{5.7}
\pi^u(f_u-f_u')\equiv\pi^{u+v}(b_{uv}-b_{uv}')f_{u+v+1}\pmod{\langle\pi^{u+v+1}\rangle}.
\end{equation}
Since $(\overline{b_{uv}}-\overline{b_{uv}'})\overline{f_{u+v+1}}\ne 0$ and $\deg (\overline{b_{uv}}-\overline{b_{uv}'})\overline{f_{u+v+1}}<\deg\overline{f_{u+v}}$, we have that $\pi^u(f_u-f_u')\notin\langle\pi^{u+v}f_{u+v},\dots,\pi^{n-1}f_{n-1}\rangle$. It follows that $\pi^u(f_u-f_u')\notin\langle\pi^0f_0,\dots,\pi^{n-1}f_{n-1}\rangle$. (Otherwise, since $\pi^u(f_u-f_u')\in\langle\pi^{u+v}\rangle$, we have $\pi^u(f_u-f_u')\in\langle\pi^0f_0,\dots,\pi^{n-1}f_{n-1}\rangle\cap \langle\pi^{u+v}\rangle=\langle\pi^{u+v}f_{u+v},\dots,\pi^{n-1}f_{n-1}\rangle$, which is a contradiction.) Then $\pi^uf_u'\notin\langle\pi^0f_0,\dots,\pi^{n-1}f_{n-1}\rangle$, and hence $\langle\pi^0f_0',\dots,\pi^{n-1}f_{n-1}'\rangle\ne \langle\pi^0f_0,\dots,\pi^{n-1}f_{n-1}\rangle$.
\end{proof}

\begin{exmp}\label{E5.2}\rm 
Let $R=\Bbb Z_4[\pi]/\langle\pi^2-2, 2\pi\rangle$. This is a finite commutative chain ring with $\text{rad}\,R=\pi R$, where $\pi$ is of nilpotency $3$, and the residue field $F=R/\text{rad}\,\pi R\cong\Bbb F_2$. Let $\alpha=x+1$, $\beta=x^2+x+1\in \Bbb F_2[x]$, which are irreducible. We will determine all ideals of $R[x]$ with invariant sequence $A\prec(\alpha^2\beta,\alpha,\alpha)$, and we will determine the partial order (inclusion) among all these ideals.

The partial order $\prec$ on $\{A\in\mathcal A:A\prec(\alpha^2\beta,\alpha,\alpha)\}$ is depicted in Figure~\ref{F1}. For each $A\prec(\alpha^2\beta,\alpha,\alpha)$, we use \eqref{5.1} to enumerate the members of $\mathcal I^A$ in terms their canonical sequences. The members of $\mathcal I^A$ are denoted by $I^A_{(\cdots)}$, where $(\cdots)$ are the parameters that appear in the canonical sequence of the ideal. See Table~\ref{Tb1} for the details. 

\begin{figure}
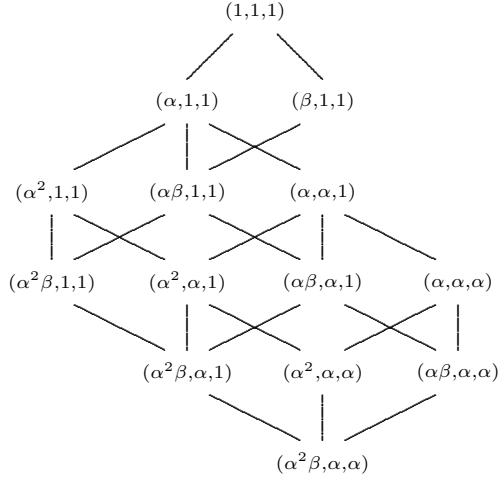

\[
\beginpicture
\setcoordinatesystem units <3mm,3mm> point at 0 0

\setlinear
\plot 11 1  7 3  /
\plot 12 1  12 3 /
\plot 13 1  17 3 /
\plot 5 5  1 7  /
\plot 6 5  6 7  /
\plot 7 5  11 7  /
\plot 11 5  7 7  /
\plot 13 5  17 7  /
\plot 17 5  13 7  /
\plot 18 5  18 7  /
\plot 0 9  0 11  /
\plot 1 9  5 11  /
\plot 5 9  1 11 /
\plot 7 9  11 11  /
\plot 11 9  7 11  /
\plot 12 9  12 11  /
\plot 17 9  13 11  /
\plot 1 13  5 15  /
\plot 6 13  6 15  /
\plot 7 13  11 15  /
\plot 11 13  7 15 /
\plot 6 17  8 19  /
\plot 12 17  10 19  /

\put {$\scriptstyle(\alpha^2\beta,\alpha,\alpha)$} at 12 0
\put {$\scriptstyle(\alpha^2\beta,\alpha,1)$} at 6 4
\put {$\scriptstyle(\alpha^2,\alpha,\alpha)$} at 12 4
\put {$\scriptstyle(\alpha\beta,\alpha,\alpha)$} at 18 4
\put {$\scriptstyle(\alpha^2\beta,1,1)$} at 0 8
\put {$\scriptstyle(\alpha^2,\alpha,1)$} at 6 8
\put {$\scriptstyle(\alpha\beta,\alpha,1)$} at 12 8
\put {$\scriptstyle(\alpha,\alpha,\alpha)$} at 18 8
\put {$\scriptstyle(\alpha^2,1,1)$} at 0 12
\put {$\scriptstyle(\alpha\beta,1,1)$} at 6 12
\put {$\scriptstyle(\alpha,\alpha,1)$} at 12 12
\put {$\scriptstyle(\alpha,1,1)$} at 6 16
\put {$\scriptstyle(\beta,1,1)$} at 12 16
\put {$\scriptstyle(1,1,1)$} at 9 20

\endpicture
\]
\caption{The partial order $\prec$ on $\{A\in\mathcal A:A\prec(\alpha^2\beta,\alpha,\alpha)\}$}\label{F1}
\end{figure}

\begin{table}
\caption{Canonical sequences of $I^A_{(\cdots)}$, $A\prec (\alpha^2\beta,\alpha,\alpha)$}\label{Tb1}
\begin{tabular}{c|c|c|l}
\hline
inv. sequence& parameters & \multicolumn{2}{c}{canonical sequence}\\ \hline
& & $f_2$ & $x+1$\\ \cline{3-4}
$(\alpha^2\beta,\alpha,\alpha)$ & $b_{11},\; b_{02}\in\{0,1\}$ & $f_1$ & $x+1+\pi b_{11}$ \\ \cline{3-4}
& & $f_0$ & $(x+1)(x^2+x+1)(x+1+\pi b_{11})+\pi^2 b_{02}$ \\ \hline
& & $f_2$ & $1$\\ \cline{3-4}
$(\alpha^2\beta,\alpha,1)$ & $b_{01}\in\{0,1\}$ & $f_1$ & $x+1$ \\ \cline{3-4}
& & $f_0$ & $(x+1)^2(x^2+x+1)+\pi b_{01}$ \\ \hline
& & $f_2$ & $x+1$\\ \cline{3-4}
$(\alpha^2,\alpha,\alpha)$ & $b_{11},\;b_{02}\in\{0,1\}$ & $f_1$ & $x+1+\pi b_{11}$ \\ \cline{3-4}
& & $f_0$ & $(x+1)(x+1+\pi b_{11})+\pi^2 b_{02}$ \\ \hline
& & $f_2$ & $x+1$\\ \cline{3-4}
$(\alpha\beta,\alpha,\alpha)$ & $b_{11},\;b_{02}\in\{0,1\}$ & $f_1$ & $x+1+\pi b_{11}$ \\ \cline{3-4}
& & $f_0$ & $(x^2+x+1)(x+1+\pi b_{11})+\pi^2 b_{02}$ \\ \hline
& & $f_2$ & $1$\\ \cline{3-4}
$(\alpha^2\beta,1,1)$ &  & $f_1$ & $1$ \\ \cline{3-4}
& & $f_0$ & $(x+1)^2(x^2+x+1)$ \\ \hline
& & $f_2$ & $1$\\ \cline{3-4}
$(\alpha^2,\alpha,1)$ & $b_{01}\in\{0,1\}$ & $f_1$ & $x+1$ \\ \cline{3-4}
& & $f_0$ & $(x+1)^2+\pi b_{01}$ \\ \hline
& & $f_2$ & $1$\\ \cline{3-4}
$(\alpha\beta,\alpha,1)$ & $b_{01}\in\{0,1\}$ & $f_1$ & $x+1$ \\ \cline{3-4}
& & $f_0$ & $(x+1)(x^2+x+1)+\pi b_{01}$ \\ \hline
& & $f_2$ & $x+1$\\ \cline{3-4}
$(\alpha,\alpha,\alpha)$ & $b_{11},\;b_{02}\in\{0,1\}$ & $f_1$ & $x+1+\pi b_{11}$ \\ \cline{3-4}
& & $f_0$ & $x+1+\pi b_{11}+\pi^2 b_{02}$ \\ \hline
& & $f_2$ & $1$\\ \cline{3-4}
$(\alpha^2,1,1)$ &  & $f_1$ & $1$ \\ \cline{3-4}
& & $f_0$ & $(x+1)^2$ \\ \hline
& & $f_2$ & $1$\\ \cline{3-4}
$(\alpha\beta,1,1)$ &  & $f_1$ & $1$ \\ \cline{3-4}
& & $f_0$ & $(x+1)(x^2+x+1)$ \\ \hline
& & $f_2$ & $1$\\ \cline{3-4}
$(\alpha,\alpha,1)$ & $b_{01}\in\{0,1\}$ & $f_1$ & $x+1$ \\ \cline{3-4}
& & $f_0$ & $x+1+\pi b_{01}$ \\ \hline
& & $f_2$ & $1$\\ \cline{3-4}
$(\alpha,1,1)$ &  & $f_1$ & $1$ \\ \cline{3-4}
& & $f_0$ & $x+1$ \\ \hline
& & $f_2$ & $1$\\ \cline{3-4}
$(\beta,1,1)$ &  & $f_1$ & $1$ \\ \cline{3-4}
& & $f_0$ & $x^2+x+1$ \\ \hline
& & $f_2$ & $1$\\ \cline{3-4}
$(1,1,1)$ &  & $f_1$ & $1$ \\ \cline{3-4}
& & $f_0$ & $1$ \\ \hline
\end{tabular}
\end{table}

There are 30 ideals of $R[x]$ with invariant sequence $A\prec(\alpha^2\beta,\alpha,\alpha)$. To determine the inclusion relations among these ideals, we only have to determine when $I^A_{(\cdots)}\subset I^B_{(\cdots)}$, where $A\prec(\alpha^2\beta,\alpha,\alpha)$ and $A\boxsucc\,B$. The necessary and sufficient condition for $I^A_{(\cdots)}\subset I^B_{(\cdots)}$ in terms of the parameters is easily found by reducing the generators of $I^A_{(\cdots)}$ by the canonical sequences of $I^B_{(\cdots)}$; see \eqref{3.3}. For example, we have 
\[
I^{(\alpha^2\beta,\alpha,\alpha)}_{b_{11},b_{02}}\subset I^{(\alpha^2\beta,\alpha,1)}_{b_{01}'} \Leftrightarrow b_{01}'=0.
\]
The diagram of the inclusion relations among all members of $\bigcup_{A\prec(\alpha^2\beta,\alpha,\alpha)}\mathcal I^A$ is obtained as follows: Replace each $A$ in Figure~\ref{F1} with the members of $\mathcal I^A$ and replace each edge going up from $A$ to $B$ in Figure~\ref{F1} with suitable edges from $\mathcal I^A$ to $\mathcal I^B$ using the necessary and sufficient condition for $I^A_{(\cdots)}\subset I^B_{(\cdots)}$. For example, the edge 
\[
\begin{array}{c}
(\alpha^2\beta,\alpha,1)\cr
\Big\vert\cr
(\alpha^2\beta,\alpha,\alpha)
\end{array}
\]
in Figure~\ref{F1} is replaced with the diagram in Figure~\ref{F2}. The complete diagram of $\bigl(\bigcup_{A\prec(\alpha^2\beta,\alpha,\alpha)}\mathcal I^A,\;\subset\bigr)$ is given in Figure~\ref{F3}.  

\begin{figure}[h]
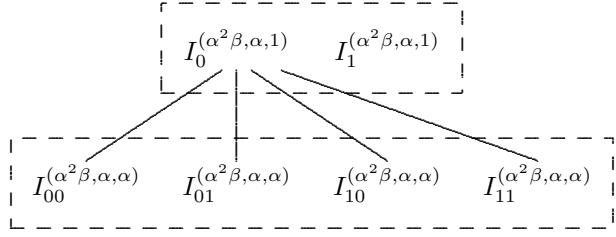

\[
\beginpicture
\setcoordinatesystem units <4mm,3mm> point at 0 0

\setlinear
\plot 2.5 3  7 7  /
\plot 7.5 3  7.5 7 /
\plot 12.5 3  8 7 /
\plot 17.5 3  9 7  /
\setdashes
\plot 0 0  20 0  20 4  0 4  0 0  /
\plot 5 6  15 6  15 10  5 10  5 6  /

\put {$ I_{00}^{(\alpha^2\beta,\alpha,\alpha)}$} at 2.5 2 
\put {$ I_{01}^{(\alpha^2\beta,\alpha,\alpha)}$} at 7.5 2 
\put {$ I_{10}^{(\alpha^2\beta,\alpha,\alpha)}$} at 12.5 2 
\put {$ I_{11}^{(\alpha^2\beta,\alpha,\alpha)}$} at 17.5 2 
\put {$ I_0^{(\alpha^2\beta,\alpha,1)}$} at 7.5 8
\put {$ I_1^{(\alpha^2\beta,\alpha,1)}$} at 12.5 8
\endpicture
\]
\caption{Ideals in $\mathcal I^{(\alpha^2\beta,\alpha,\alpha)}$ and $\mathcal I^{(\alpha^2\beta,\alpha,1)}$}\label{F2}
\end{figure}
\end{exmp}

\newpage
\begin{figure}
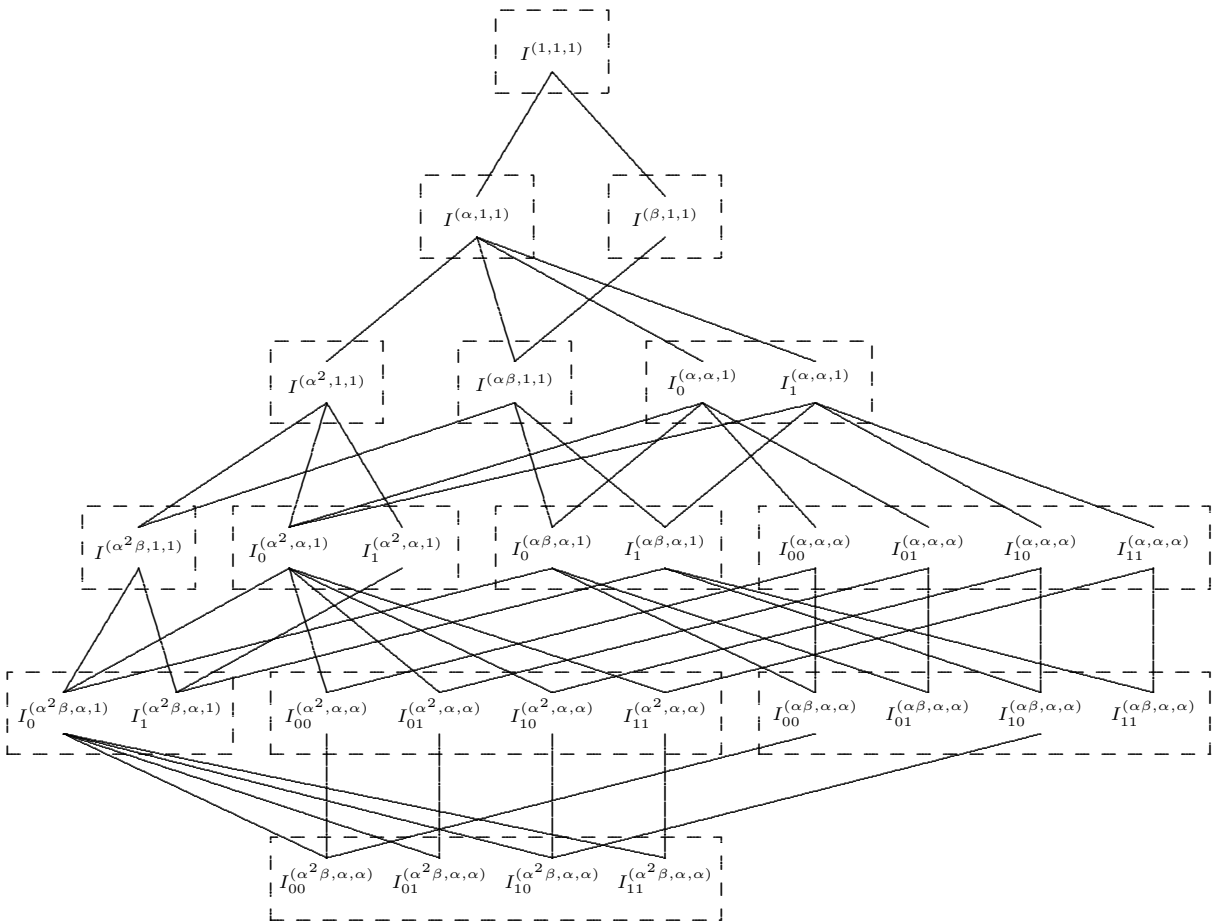

\vskip5mm
\begin{sideways}
\beginpicture
\setcoordinatesystem units <5mm,5.5mm> point at 0 0

\setlinear
\plot 9 1.5  2 4.5 /
\plot 9 1.5  9 4.5  /
\plot 9 1.5  22 4.5 /
\plot 12 1.5  2 4.5 /
\plot 12 1.5  12 4.5 /
\plot 15 1.5  2 4.5 /
\plot 15 1.5  15 4.5 /
\plot 15 1.5  28 4.5 /
\plot 18 1.5  2 4.5 /
\plot 18 1.5  18 4.5 /
\plot 2 5.5  4 8.5 /
\plot 2 5.5  8 8.5 /
\plot 2 5.5  15 8.5 /
\plot 5 5.5  4 8.5 /
\plot 5 5.5  11 8.5 /
\plot 5 5.5  18 8.5 /
\plot 9 5.5  8 8.5 /
\plot 9 5.5  22 8.5 /
\plot 12 5.5  8 8.5 /
\plot 12 5.5  25 8.5 /
\plot 15 5.5  8 8.5 /
\plot 15 5.5  28 8.5 /
\plot 18 5.5  8 8.5 /
\plot 18 5.5  31 8.5 /
\plot 22 5.5  15 8.5 /
\plot 22 5.5  22 8.5 /
\plot 25 5.5  15 8.5 /
\plot 25 5.5  25 8.5 /
\plot 28 5.5  18 8.5 /
\plot 28 5.5  28 8.5 /
\plot 31 5.5  18 8.5 /
\plot 31 5.5  31 8.5 /
\plot 4 9.5  9 12.5 /
\plot 4 9.5  14 12.5 /
\plot 8 9.5  9 12.5 /
\plot 8 9.5  19 12.5 /
\plot 8 9.5  22 12.5 /
\plot 11 9.5  9 12.5 /
\plot 15 9.5  14 12.5 /
\plot 15 9.5  19 12.5 /
\plot 18 9.5  14 12.5 /
\plot 18 9.5  22 12.5 /
\plot 22 9.5  19 12.5 /
\plot 25 9.5  19 12.5 /
\plot 28 9.5  22 12.5 /
\plot 31 9.5  22 12.5 /
\plot 9 13.5  13 16.5 /
\plot 14 13.5  13 16.5 /
\plot 14 13.5  18 16.5 /
\plot 19 13.5  13 16.5 /
\plot 22 13.5  13 16.5 /
\plot 13 17.5  15 20.5 /
\plot 18 17.5  15 20.5 /

\setdashes
\plot 7.5 0  19.5 0  19.5 2  7.5 2  7.5 0  /
\plot 0.5 4  6.5 4  6.5 6  0.5 6  0.5 4  /
\plot 7.5 4  19.5 4  19.5 6  7.5 6  7.5 4  /
\plot 20.5 4  32.5 4  32.5 6  20.5 6  20.5 4  /
\plot 2.5 8  5.5 8  5.5 10  2.5 10  2.5 8  /
\plot 6.5 8  12.5 8  12.5 10  6.5 10  6.5 8  /
\plot 13.5 8  19.5 8  19.5 10  13.5 10  13.5 8  /
\plot 20.5 8  32.5 8  32.5 10  20.5 10  20.5 8  /
\plot 7.5 12  10.5 12  10.5 14  7.5 14  7.5 12  /
\plot 12.5 12  15.5 12  15.5 14  12.5 14  12.5 12  /
\plot 17.5 12  23.5 12  23.5 14  17.5 14  17.5 12  /
\plot 11.5 16  14.5 16  14.5 18  11.5 18  11.5 16  /
\plot 16.5 16  19.5 16  19.5 18  16.5 18  16.5 16  /
\plot 13.5 20  16.5 20  16.5 22  13.5 22  13.5 20  /
 
\put {$\scriptstyle I_{00}^{(\alpha^2\beta,\alpha,\alpha)}$} at 9 1 
\put {$\scriptstyle I_{01}^{(\alpha^2\beta,\alpha,\alpha)}$} at 12 1 
\put {$\scriptstyle I_{10}^{(\alpha^2\beta,\alpha,\alpha)}$} at 15 1 
\put {$\scriptstyle I_{11}^{(\alpha^2\beta,\alpha,\alpha)}$} at 18 1 
\put {$\scriptstyle I_0^{(\alpha^2\beta,\alpha,1)}$} at 2 5 
\put {$\scriptstyle I_1^{(\alpha^2\beta,\alpha,1)}$} at 5 5
\put {$\scriptstyle I_{00}^{(\alpha^2,\alpha,\alpha)}$} at 9 5
\put {$\scriptstyle I_{01}^{(\alpha^2,\alpha,\alpha)}$} at 12 5
\put {$\scriptstyle I_{10}^{(\alpha^2,\alpha,\alpha)}$} at 15 5
\put {$\scriptstyle I_{11}^{(\alpha^2,\alpha,\alpha)}$} at 18 5
\put {$\scriptstyle I_{00}^{(\alpha\beta,\alpha,\alpha)}$} at 22 5 
\put {$\scriptstyle I_{01}^{(\alpha\beta,\alpha,\alpha)}$} at 25 5
\put {$\scriptstyle I_{10}^{(\alpha\beta,\alpha,\alpha)}$} at 28 5
\put {$\scriptstyle I_{11}^{(\alpha\beta,\alpha,\alpha)}$} at 31 5
\put {$\scriptstyle I^{(\alpha^2\beta,1,1)}$} at 4 9
\put {$\scriptstyle I_0^{(\alpha^2,\alpha,1)}$} at 8 9
\put {$\scriptstyle I_1^{(\alpha^2,\alpha,1)}$} at 11 9
\put {$\scriptstyle I_0^{(\alpha\beta,\alpha,1)}$} at 15 9
\put {$\scriptstyle I_1^{(\alpha\beta,\alpha,1)}$} at 18 9
\put {$\scriptstyle I_{00}^{(\alpha,\alpha,\alpha)}$} at 22 9
\put {$\scriptstyle I_{01}^{(\alpha,\alpha,\alpha)}$} at 25 9
\put {$\scriptstyle I_{10}^{(\alpha,\alpha,\alpha)}$} at 28 9
\put {$\scriptstyle I_{11}^{(\alpha,\alpha,\alpha)}$} at 31 9
\put {$\scriptstyle I^{(\alpha^2,1,1)}$} at 9 13
\put {$\scriptstyle I^{(\alpha\beta,1,1)}$} at 14 13
\put {$\scriptstyle I_0^{(\alpha,\alpha,1)}$} at 19 13
\put {$\scriptstyle I_1^{(\alpha,\alpha,1)}$} at 22 13
\put {$\scriptstyle I^{(\alpha,1,1)}$} at 13 17
\put {$\scriptstyle I^{(\beta,1,1)}$} at 18 17
\put {$\scriptstyle I^{(1,1,1)}$} at 15 21
\endpicture
\end{sideways}
\caption{The poset $\bigl(\bigcup_{A\prec(\alpha^2\beta,\alpha,\alpha)}\mathcal I^A,\;\subset\bigr)$}\label{F3}
\end{figure}


\section{Algorithms for Questions~\ref{Q1.1} and \ref{Q1.2}}

\subsection{The master algorithm for Question~\ref{Q1.1}}\ 

\begin{algo}\label{A6.1}\rm (Master algorithm for Question~\ref{Q1.1})
Let $I$ be a proper ideal of $R[x]$ with canonical sequence $(f_0,\dots,f_{n-1})$. Then $\overline{f_0}\ne 1$. If $\overline{f_0}=0$, $R[x]/I$ is not artinian and hence cannot be Frobenius. On the other hand, if $\overline{f_0}\ne 0$, $R[x]/I$ is artinian. Therefore assume that $\overline{f_0}\ne 0$. 
Write $\overline{f_0}=\overline{p_1}^{e_1}\cdots\overline{p_k}^{e_k}$, where $e_i>0$ and $p_i\in R[x]$ such that $\overline{p_1},\dots,\overline{p_k}\in F[x]$ are irreducible and distinct. By Proposition~\ref{P4.1}, the maximal ideals of $R[x]/I$ are precisely $\langle p_i,\pi\rangle/I$, $1\le i\le k$. We have 
\[
\text{ann}(\langle p_i,\pi\rangle/I)=J_i/I,
\]
where
\[
J_i=[I:p_i]\cap[I:\pi].
\]
The canonical sequence of $[I:\pi]$ is $(f_1,\dots,f_{n-1},1)$. The canonical sequence of $[I:p_i]$ is computed using Algorithm~\ref{A6.2}. Then the canonical sequence of $J_i$ is computed using Algorithm~\ref{A6.3}, and the canonical sequence gives the invariant sequence $A_i$ of $J_i$. By Theorem~\ref{T2.3} and Proposition~\ref{P4.3}, $R[x]/I$ is Frobenius if and only if $A_i\boxprec\,(\overline{f_0},\dots,\overline{f_{n-1}})$ for all $1\le i\le k$.
\end{algo}

\subsection{Supporting algorithms}\

\begin{algo}\label{A6.2}\rm (Canonical sequence of $[I:p]$) Let $I$ be an ideal of $R[x]$ with canonical sequence $(f_0,\dots,f_{n-1})$ and let $0\ne p\in R[x]$. The following algorithm produces a canonical sequence $(h_0,\dots,h_{n-1})$ of $[I:p]$. 

First write $p=\pi^kq$, where $0\le k\le n-1$ and $q\in R[x]$, $\overline q\ne 0$. Set $h_i=1$ for $n-k\le i\le n-1$. 

For each $0\le i<n-k$, if $\overline{f_{i+k}}=0$, choose $h_i=0$. If $\overline{f_{i+k}}\ne 0$, compute $a_0,b_0\in R[x]$ such that
\[
\overline{a_0}=\frac{\overline{f_{i+k}}}{\text{gcd}(\overline q,\overline{f_{i+k}})},\qquad \overline{b_0}=\frac{\overline q}{\text{gcd}(\overline q,\overline{f_{i+k}})}.
\]
Then $\overline{a_0}\,\overline q$ is the smallest degree multiple of $\overline q$ which is also a multiple of $\overline{f_{i+k}}$. Since $\overline{a_0}\,\overline q=\overline{b_0}\,\overline{f_{i+k}}$, we have 
\begin{equation}\label{6.1}
a_0q=b_0f_{i+k}+\pi c_1
\end{equation}
for some $c_1\in R[x]$. Compute $d_1\in R[x]$ such that 
\[
\overline{d_1}=\frac{\text{gcd}(\overline q,\overline{f_{i+k+1}})}{\text{gcd}(\overline q,\overline{f_{i+k+1}},\overline{c_1})}.
\]
Then $\overline{d_1}\,\overline{c_1}$ is the smallest degree multiple of $\overline{c_1}$ which is also a multiple of $\text{gcd}(\overline q,\overline{f_{i+k+1}})$. Compute $a_1,b_1\in R[x]$ such that $\overline{d_1}\,\overline{c_1}=-\overline{a_1}\,\overline q+\overline{b_1}\,\overline{f_{i+k+1}}$ and write $d_1c_1=-a_1q+b_1f_{i+k+1}+\pi c_2$, where $c_2\in R[x]$. Multiplying \eqref{6.1} by $d_1$, we have
\[
(d_1a_0+\pi a_1)q=d_1b_0f_{i+k}+\pi b_1f_{i+k+1}+\pi^2c_2.
\]
Recursively, we compute $d_j,a_j,b_j,c_{j+1}$ ($1\le j\le n-k-i-1$) such that 
\[
\overline{d_j}=\frac{\text{gcd}(\overline q,\overline{f_{i+k+j}})}{\text{gcd}(\overline q,\overline{f_{i+k+j}},\overline{c_j})}
\]
and 
\[
d_jc_j=-a_jq+b_jf_{i+k+j}+\pi c_{j+1}.
\]
Finally, we have
\begin{equation}\label{6.2}
\begin{split}
&(d_{n-k-i-1}\cdots d_2d_1a_0+\pi d_{n-k-i-1}\cdots d_2a_1+\cdots+\pi^{n-k-i-1}a_{n-k-i-1})q\cr
=\,&d_{n-k-i-1}\cdots d_2d_1b_0f_{i+k}+\pi d_{n-k-i-1}\cdots d_2b_1f_{i+k+1}+\cdots+\pi^{n-k-i-1}b_{n-k-i-1}f_{n-1}\cr
&+\pi^{n-k-i}c_{n-k-i}.
\end{split}
\end{equation}
Let
\[
h_i=d_{n-k-i-1}\cdots d_2d_1a_0+\pi d_{n-k-i-1}\cdots d_2a_1+\cdots+\pi^{n-k-i-1}a_{n-k-i-1}.
\]
By \eqref{6.2}, $\pi^ih_ip=\pi^ih_i\pi^kq\in I$. From the above construction, it is not difficult to show that among all $h\in R[x]$ satisfying $\pi^ihp\in I$, $h_i$ is such that $\deg\overline{h_i}$ is the smallest.

Therefore, $(h_0,\dots,h_{n-1})$ is a canonical sequence of $[I:p]$.
\end{algo}

\begin{algo}\label{A6.3}\rm (Canonical sequence of $I\cap J$)
Let $I$ and $J$ be two ideals of $R[x]$ with canonical sequences $(f_0,\dots,f_{n-1})$ and $(g_0,\dots,g_{n-1})$, respectively. A canonical sequence $(h_0,\dots,h_{n-1})$ of $I\cap J$ can be found by a method almost identical to that of Algorithm~\ref{A6.2}.

For each $0\le i\le n-1$, if $\overline{f_i}=\overline{g_i}=0$, choose $h_i=0$. If $(\overline{f_i},\overline{g_i})\ne (0,0)$, say $\overline{f_i}\ne 0$, compute $a_0,b_0\in R[x]$ such that 
\[
\overline{a_0}=\frac{\overline{g_i}}{\text{gcd}(\overline{f_i},\overline{g_i})},\qquad \overline{b_0}=\frac{\overline{f_i}}{\text{gcd}(\overline{f_i},\overline{g_i})},
\]
and write
\begin{equation}\label{6.3}
a_0f_i-b_0g_i=\pi c_1
\end{equation}
for some $c_1\in R[x]$. Compute $d_1\in R[x]$ such that 
\[
\overline{d_1}=\frac{\text{gcd}(\overline{f_{i+1}},\overline{g_{i+1}})}{\text{gcd}(\overline{f_{i+1}},\overline{g_{i+1}},\overline{c_1})}.
\]
Then $\overline{d_1}\overline{c_1}$ is the smallest degree multiple of $\overline{c_1}$ which is also a multiple of $\text{gcd}(\overline{f_{i+1}},\overline{g_{i+1}})$. Compute $a_1,b_1\in R[x]$ such that $\overline{d_1}\overline{c_1}=-\overline{a_1}\overline{f_{i+1}}+\overline{b_1}\overline{g_{i+1}}$, and write $d_1c_1=-a_1f_{i+1}+b_1g_{i+1}+\pi c_2$, where $c_2\in R[x]$. Multiplying \eqref{6.3} by $d_1$ gives
\[
(d_1a_0f_i+\pi a_1f_{i+1})-(d_1b_0g_i+\pi b_1g_{i+1})=\pi^2 c_2.
\]  
Recursively, we compute $d_j,a_j,b_j,c_{j+1}$ ($1\le j\le n-i-1$) such that 
\[
\overline{d_j}=\frac{\text{gcd}(\overline{f_{i+j}},\overline{g_{i+j}})}{\text{gcd}(\overline{f_{i+j}},\overline{g_{i+j}},\overline{c_j})}
\]
and 
\[
d_jc_j=-a_jf_{i+j}+b_jg_{i+j}+\pi c_{j+1}.
\]
Finally, we have
\begin{equation}\label{6.4}
\begin{split}
&(d_{n-i-1}\cdots d_2d_1a_0f_i+\pi d_{n-i-1}\cdots d_2a_1f_{i+1}+\cdots+\pi^{n-i-1}a_{n-i-1}f_{n-1})\cr
-\,&(d_{n-i-1}\cdots d_2d_1b_0g_i+\pi d_{n-i-1}\cdots d_2b_1g_{i+1}+\cdots+\pi^{n-i-1}b_{n-i-1}g_{n-1})=0.
\end{split}
\end{equation}
Let
\[
h_i=d_{n-i-1}\cdots d_2d_1a_0f_i+\pi d_{n-i-1}\cdots d_2a_1f_{i+1}+\cdots+\pi^{n-i-1}a_{n-i-1}f_{n-1}.
\] 
It is clear from \eqref{6.4} that $\pi^ih_i\in I\cap J$. From the above construction, it can be shown that among all $h\in R[x]$ with $\pi^i h\in I\cap J$, $h_i$ is such that $\deg\overline{h_i}$ is the smallest.

Therefore, $(h_0,\dots,h_{n-1})$ is a canonical sequence of $I\cap J$.
\end{algo}
 
\subsection{The master algorithm for Question~\ref{Q1.2}}

\begin{prop}\label{P6.4}
Let $I$ be a proper ideal of $R[x]$ with invariant sequence $A$. Then $R[x]/I$ is local if and only if $A=(\alpha^{e_0},\dots,\alpha^{e_{n-1}})$, where $\alpha\in F[x]$ is irreducible, $e_0\ge \cdots\ge e_{n-1}\ge 0$, and $e_0>0$.
\end{prop}

\begin{proof} This follows immediately from Proposition~\ref{P4.1}.
\end{proof}

\begin{prop}\label{P6.5}
Let $I$ be an ideal of $R$ with invariant sequence $(\alpha^{e_0},\dots,\alpha^{e_i},1,\dots,1)$, where $\alpha\in F[x]$ is irreducible and $e_i>0$. 
\begin{itemize}
  \item [(i)] There is a unique ideal containing $I$ with invariant sequence $(\alpha^{e_0},\dots,\alpha^{e_{i-1}},$ $\alpha^{e_i-1},1,\dots,1)$.
  \item [(ii)] $R[x]/I$ is Frobenius and local if and only for each $0\le j\le i-1$ with $e_j>e_{j+1}$, there is no ideal containing $I$ with invariant sequence $(\alpha^{e_0},\dots,\alpha^{e_{j-1}},$ $\alpha^{e_j-1},\alpha^{e_{j+1}},\dots,\alpha^{e_i},1,\dots,1)$.
\end{itemize} 
\end{prop}

\begin{proof}
(i) We have $I\cap\langle\pi^{i+1}\rangle =\langle\pi^{i+1}\rangle$. Let $(f_0,\dots,f_{n-1})$ be a canonical sequence of $I$. Let $g\in R[x]$ be such that $\overline g=\alpha^{e_i-1}$. Then condition (i) -- (iii) in Theorem~\ref{T4.5} are satisfied with $\overline r=\alpha$. (In fact, $\pi^{i+1}g,\pi^i(f_i-rg)\in\langle\pi^{i+1}\rangle\subset I$, which gives (ii) and (iii).) By Theorem~\ref{T4.5}, $I+\langle\pi^ig\rangle$ has invariant sequence $(\alpha^{e_0},\dots,\alpha^{e_{i-1}},\alpha^{e_i-1},1,\dots,1)$. If $g'\in R[x]$ is another polynomial satisfying $\overline{g'}=\alpha^{e_i-1}$, then $\pi^i(g-g')\in\langle\pi^{i+1}\rangle\subset I$, and hence $I+\langle\pi^ig\rangle=I+\langle\pi^i g'\rangle$. Thus by Theorem~\ref{T4.5}, $I+\langle\pi^ig\rangle$ is the unique ideal containing $I$ with invariant sequence $(\alpha^{e_0},\dots,\alpha^{e_{i-1}},\alpha^{e_i-1},1,\dots,1)$.

\medskip
(ii) The claim follows from (i) and Theorem~\ref{T2.4}
\end{proof}

Combining Proposition~\ref{P6.5} (ii) and Theorem~\ref{T4.6} gives the following result.

\begin{cor}\label{C6.6}
Let $I$ be an ideal of $R[x]$ with canonical sequence $(f_0,\dots, f_{n-1})$ and invariant sequence $(\alpha^{e_0},\dots,\alpha^{e_i},1,\dots,1)$, where $\alpha\in F[x]$ is irreducible and $e_i>0$. Let $r\in R[x]$ be such that $\overline r=\alpha$.  Then $R[x]/I$ is Frobenius and local if and only if for each $0\le j\le i-1$ with $e_j>e_{j+1}$,
\[
\pi^j(f_j-r^{e_j-e_{j+1}}f_{j+1})\notin I\cap \langle\pi^{i+1}\rangle+r[I\cap\langle\pi^{i+2}\rangle:\pi].
\]
\end{cor}

The master algorithm for answering Question~\ref{Q1.2} now becomes obvious.

\begin{algo}\label{A6.7}\rm (Master algorithm for Question~\ref{Q1.2})
Let $I$ be an ideal of $R[x]$ with canonical sequence $(f_0,\dots,f_{n-1})$ such that $(\overline{f_0},\dots,\overline{f_{n-1}})=(\alpha^{e_0},\dots,\alpha^{e_i},1,\dots,1)$, where $\alpha\in F[x]$ is irreducible and $e_i>0$. We first compute $r\in R[x]$ such that $\overline r=\alpha$ and determine 
\[
\Lambda=\{j:0\le j\le i-1:e_j>e_{j+1}\}.
\]
For each $j\in\Lambda$, we compute 
\[
v_j=\pi^j(f_j-r^{e_j-e_{j+1}}f_{j+1}).
\]
The ideal $I\cap\langle\pi^{j+1}\rangle$ has a canonical sequence $(0,\dots,0,f_{j+1},\dots,f_{n-1})$; the ideal $r[I\cap\langle\pi^{j+2}\rangle:\pi]$ has a canonical sequence $(0,\dots,0,rf_{j+2},\dots,rf_{n-1},r)$. We compute the canonical sequence of 
\[
\begin{split}
K_j:=\,& I\cap\langle\pi^{j+1}\rangle+r[I\cap\langle\pi^{j+2}\rangle:\pi]\cr
=\,&\langle \pi^{j+1}f_{j+1},\cdots,\pi^{n-1}f_{n-1}\rangle+\langle \pi^{j+1}rf_{j+2},\cdots,\pi^{n-2}rf_{n-1},\pi^{n-1}r\rangle
\end{split}
\]
using Algorithm~\ref{A3.3}, and determine if $v_j$ belongs to $K_j$ by computing the reduction of $v_j$ by the canonical sequence of $K_j$. The ring $R[x]/I$ is Frobenius and local if and only if $v_j\notin K_j$ for all $j\in\Lambda$. 
\end{algo}    

\begin{exmp}\label{E6.3}\rm 
We revisit the following result that appeared in \cite[Example 2.6]{Hou03}. Let $p$ be a prime and $a,b,m,n$ be integers such that $n\ge 2$, $0<a<b\le m$, $2b-a\ge m$. Then $\Bbb Z_{p^m}[x]/\langle x^n,p^ax^{n-1}-p^b\rangle$ is Frobenius and local if and only if $2b-a=m$. (There is no need to consider the case $2b-a<m$. Since $p^bx\in\langle x^n,p^ax^{n-1}-p^b\rangle$, we have $p^{2b-a}=-p^{b-a}(p^ax^{n-1}-p^b)+p^bx^{n-1}\in \langle x^n,p^ax^{n-1}-p^b\rangle$. If $2b-a<m$, then $\Bbb Z_{p^m}[x]/\langle x^n,p^ax^{n-1}-p^b\rangle\cong \Bbb Z_{p^{2b-a}}[x]/\langle x^n,p^ax^{n-1}-p^b\rangle$.) We use the above algorithm to give another proof of this result.

Let $R=\Bbb Z_{p^m}$ ($\pi=p$, which has nilpotency $m$) and $I= \langle x^n,p^ax^{n-1}-p^b\rangle$. The canonical sequence of $I$ is 
\[
(f_0,\dots,f_{m-1})=(\underbrace{x^n,\dots,x^n}_a,\underbrace{x^{n-1}-p^{b-a},\dots,x^{n-1}-p^{b-a}}_{b-a},\underbrace{x,\dots,x}_{m-b}),
\]
and we have
\[
(\overline{f_0},\dots,\overline{f_{m-1}})=(\underbrace{x^n,\dots,x^n}_a,\underbrace{x^{n-1},\dots,x^{n-1}}_{b-a},\underbrace{x,\dots,x}_{m-b}).
\]
We choose $r=x$. Clearly
\[
\Lambda=
\begin{cases}
\{a-1\}&\text{if $n=2$ or $m=b$},\cr
\{a-1,b-1\}&\text{if $n>2$ and $m>b$}.
\end{cases}
\]
We have $v_{a-1}=p^{a-1}(f_{a-1}-xf_a)=p^{a-1}[x^n-x(x^{n-1}-p^{b-a})]=p^{b-1}x$. 
When $n>2$ and $m>b$, we have $v_{b-1}=p^{b-1}(f_{b-1}-x^{n-2}f_b)=p^{b-1}(x^{n-1}-p^{b-a}-x^{n-1})=-p^{2b-a-1}$.
It is easy to see that
\begin{equation}\label{6.5}
\begin{split}
K_{a-1}=\,&\langle p^a(x^{n-1}-p^{b-a}),\dots,p^{b-1}(x^{n-1}-p^{b-a}),p^bx,\dots,p^{m-1}x\rangle\cr
&+\langle p^ax(x^{n-1}-p^{b-a}),\dots,p^{b-2}x(x^{n-1}-p^{b-a}),p^{b-1}x^2,\dots,p^{m-2}x^2,p^{m-1}x\rangle\cr
=\,&\langle p^a(x^{n-1}-p^{b-a}),\dots,p^{b-2}(x^{n-1}-p^{b-a})\rangle\cr
&+
\begin{cases}
\langle p^{m-1}x\rangle&\text{if\ $m=b$},\vspace{1mm}\cr
\langle p^{b-1}(x-p^{b-a}),p^bx,\dots,p^{m-1}x\rangle&\text{if\ $n=2,\ m>b$},\vspace{1mm}\cr
\langle p^{b-1}x^2,p^bx,\dots,p^{m-1}x\rangle&\text{if\ $n>2,\ m>b$},\cr
&\kern 3mm 2b-a>m,\vspace{1mm}\cr
\langle p^{b-1}x^2,p^bx,\dots,p^{2b-a-2}x,p^{2b-a-1},\dots,p^{m-1}\rangle&\text{if\ $n>2,\ m>b$},\cr
&\kern 3mm 2b-a\le m.
\end{cases}
\end{split}
\end{equation}
The canonical sequence of $K_{a-1}$ is
\[
\begin{cases}
(\underbrace{0,\dots,0}_a,\underbrace{x^{n-1}-p^{b-a},\dots,x^{n-1}-p^{b-a}}_{b-a-1},x)&\text{if\ $m=b$},\vspace{1mm}\cr
(\underbrace{0,\dots,0}_a,\underbrace{x-p^{b-a},\dots,x-p^{b-a}}_{b-a},\underbrace{x,\dots,x}_{m-b})&\text{if\ $n=2,\ m>b$},\vspace{1mm}\cr
(\underbrace{0,\dots,0}_a,\underbrace{x^{n-1}-p^{b-a},\dots,x^{n-1}-p^{b-a}}_{b-a-1},x^2,\underbrace{x,\dots,x}_{m-b})&\text{if\ $n>2,\ m>b$},\vspace{-3mm}\cr 
&\kern 3mm 2b-a>m,\vspace{2mm}\cr
(\underbrace{0,\dots,0}_a,\underbrace{x^{n-1}-p^{b-a},\dots,x^{n-1}-p^{b-a}}_{b-a-1},x^2,\underbrace{x,\dots,x}_{b-a-1},\underbrace{1,\dots,1}_{m-2b+a+1})&\text{if\ $n>2,\ m>b$},\vspace{-3mm}\cr
&\kern 3mm 2b-a\le m.
\end{cases}
\]
When $n>2$ and $m>b$, we have 
\[
K_{b-1}=\langle p^bx,\dots,p^{m-1}x\rangle+\langle p^bx^2,\dots,p^{m-2}x^2,p^{m-1}x\rangle=\langle p^bx,\dots,p^{m-1}x\rangle,
\]
whose canonical sequence is $(\underbrace{0,\dots,0}_b,\underbrace{x,\dots,x}_{m-b})$.

First assume $2b-a=m$. Then $m>b$. We claim that $v_{a-1}=p^{b-1}x\notin K_{a-1}$. By \eqref{6.5}, the claim is obviously true when $n>2$. When $n=2$, the contrary of the claim would imply that $p^{2b-a-1}=p^{m-1}\in K_{a-1}$, which is not true. When $n>2$ and $m>b$, we have $v_{b-1}=-p^{2b-a-1}=-p^{m-1}\notin K_{b-1}$. Therefore $R[x]/I$ is Frobenius and local.

Next assume $2b-a>m$. If $m=b$, by \eqref{6.5}, we have $v_{a-1}=p^{b-1}x\in K_{a-1}$. If $n=2$ and $m>b$, by \eqref{6.5} and the condition $2b-a>m$, we also have $v_{a-1}=p^{b-1}x\in K_{a-1}$. If $n>2$ and $m>b$, we have $v_{b-1}=-p^{2b-a-1}=0\in K_{b-1}$. Thus $R[x]/I$ is local but not Frobenius.
\end{exmp}


\section{Cases of Small Nilpotency}

We apply the algorithms in Section 6 to commutative chain rings $R$ whose radicals have small nilpotency. The small nilpotency of $\text{rad}\,R$ allows us to carry out the computations in the algorithms to an extent that results in explicit answers to Questions~\ref{Q1.1} and \ref{Q1.2}.

\begin{thm}\label{T7.1}
Let $R$ be a commutative chain ring with $\text{\rm rad}\,R=\pi R$, where $\pi\in R$ is of nilpotency $2$. Let $I$ be an ideal of $R[x]$ such that $I\not\subset\langle\pi\rangle$. Then $I$ has a canonical sequence $(f_0,f_1)$ of the form 
\begin{equation}\label{7.1}
\begin{cases}
f_0=p_1^{u_1}\cdots p_k^{u_k}+\pi p_1^{w_1}\cdots p_k^{w_k}q,\cr
f_1=p_1^{v_1}\cdots p_k^{v_k},
\end{cases}
\end{equation}
where $\overline{p_1},\dots,\overline{p_k}\in F[x]$ are distinct and irreducible, $\text{\rm gcd}(\overline{p_1}\cdots\overline{p_k},\overline q)=1$, and $u_i>0$, $0\le v_i\le u_i$, $w_i\ge 0$ for all $1\le i\le k$. The invariant sequence of $[I:p_1]\cap[I:\pi]$ is
\[
\begin{cases}
(\overline{p_1}^{u_1-1}\overline{p_2}^{u_2}\cdots\overline{p_k}^{u_k},\;\overline{p_2}^{v_2}\cdots\overline{p_k}^{v_k})&\text{if}\ v_1=0,\cr
(\overline{p_1}^{u_1}\overline{p_2}^{u_2}\cdots\overline{p_k}^{u_k},\;\overline{p_1}^{v_1-1} \overline{p_2}^{v_2}\cdots\overline{p_k}^{v_k})&\text{if $v_1>0$ and $w_1$ or $u_1-v_1$ is $0$},\cr
(\overline{p_1}^{u_1-1}\overline{p_2}^{u_2}\cdots\overline{p_k}^{u_k},\;\overline{p_1}^{v_1-1} \overline{p_2}^{v_2}\cdots\overline{p_k}^{v_k})&\text{if $v_1>0$, $w_1>0$, $u_1-v_1>0$}.
\end{cases}
\]
In particular, $[I:p_1]\cap[I:\pi]/I$ is minimal if and only if at least one of $v_1,w_1,u_1-v_1$ is $0$.
\end{thm}

\begin{proof}
We only have to prove the claim about the invariant sequence of $[I:p_1]\cap[I:\pi]$. (The claim about the canonical sequence of $I$ is obviously true. The necessary and sufficient condition for $[I:p_1]\cap[I:\pi]/I$to be minimal follows from Proposition~\ref{P4.3} once the invariant sequence of $[I:p_1]\cap[I:\pi]$ is known.)

The canonical sequence of $[I:\pi]$ is $(f_1,1)$. We use Algorithm~\ref{A6.2} to compute a canonical sequence $(h_0,h_1)$ of $[I:p_1]$. We will see that the invariant sequence of $[I:p_1]\cap[I:\pi]$ follows rather easily afterwards.

\medskip
{\bf Case 1.} Assume $v_1=0$. 

We can choose $h_1=p_2^{v_2}\cdots p_k^{v_k}$. To find $h_0$, we first write
\[
(p_1^{u_1-1}p_2^{u_2}\cdots p_k^{u_k})p_1=f_0-\pi p_1^{w_1}\cdots p_k^{w_k}q.
\]
In the notation of Algorithm~\ref{A6.2}, $a_0= p_1^{u_1-1}p_2^{u_2}\cdots p_k^{u_k}$, $b_0=1$, $c_1=-p_1^{w_1}\cdots p_k^{w_k}q$, $\text{gcd}(\overline{p_1},\overline{f_1})/\text{gcd}(\overline{p_1},\overline{f_1},\overline{c_1})=1$, $d_1=1$. So we write
\[
-p_1^{w_1}\cdots p_k^{w_k}q=-a_1p_1+b_1f_1+\pi c_2,
\]
where $a_1,b_1,c_2\in R[x]$. Eq.~\eqref{6.2} appears as 
\[
(p_1^{u_1-1}p_2^{u_2}\cdots p_k^{u_k}+\pi a_1)p_1=f_0+\pi b_1f_1,
\]
so $h_0=p_1^{u_1-1}p_2^{u_2}\cdots p_k^{u_k}+\pi a_1$. Clearly, $[I:p_1]=\langle h_0,\pi h_1\rangle\subset\langle f_1,\pi\rangle=[I:\pi]$.
Hence $[I:p_1]\cap[I:\pi]=[I:p_1]$, whose invariant sequence is $(\overline{p_1}^{u_1-1}\overline{p_2}^{u_2}\cdots\overline{p_k}^{u_k},$ $\overline{p_2}^{v_2}\cdots\overline{p_k}^{v_k})$.

\medskip
{\bf Case 2.} Assume $v_1>0$ and $w_1=0$.

We can choose $h_1=p_1^{v_1-1}p_2^{v_2}\cdots p_k^{v_k}$. To find $h_0$, we write
\[
(p_1^{u_1-1}p_2^{u_2}\cdots p_k^{u_k})p_1=f_0-\pi p_2^{w_2}\cdots p_k^{w_k}q.
\]
In the notation of Algorithm~\ref{A6.2}, $a_0= p_1^{u_1-1}p_2^{u_2}\cdots p_k^{u_k}$, $b_0=1$, $c_1=-p_2^{w_2}\cdots p_k^{w_k}q$, $\text{gcd}(\overline{p_1},\overline{f_1})/\text{gcd}(\overline{p_1},\overline{f_1},\overline{c_1})=\overline{p_1}$, $d_1=p_1$. So we write
\[
p_1(-p_2^{w_2}\cdots p_k^{w_k}q)=-(p_2^{w_2}\cdots p_k^{w_k}q)p_1.
\]
(In the notation of Algorithm~\ref{A6.2}, $a_1=p_2^{w_2}\cdots p_k^{w_k}q$, $b_1=0$, $c_2=0$.) Eq.~\eqref{6.2} appears as 
\[
(p_1^{u_1}\cdots p_k^{u_k}+\pi p_2^{w_2}\cdots p_k^{w_k}q)p_1=p_1f_0,
\]
so $h_0=p_1^{u_1}\cdots p_k^{u_k}+\pi p_2^{w_2}\cdots p_k^{w_k}q=f_0$. Clearly, $[I:p_1]=\langle h_0,\pi h_1\rangle\subset \langle f_1,\pi\rangle=[I:\pi]$. Hence $[I:p_1]\cap[I:\pi]=[I:p_1]$, whose invariant sequence is $(\overline{p_1}^{u_1}\cdots\overline{p_k}^{u_k},$  $\overline{p_1}^{v_1-1} \overline{p_2}^{v_2}\cdots\overline{p_k}^{v_k})$.

\medskip
{\bf Case 3.} Assume $v_1>0$ and $w_1>0$.

We can choose $h_1=p_1^{v_1-1}p_2^{v_2}\cdots p_k^{v_k}$. To find $h_0$, write
\[
(p_1^{u_1-1}p_2^{u_2}\cdots p_k^{u_k})p_1=f_0-\pi p_1^{w_1}\cdots p_k^{w_k}q.
\]
In the notation of Algorithm~\ref{A6.2}, $a_0= p_1^{u_1-1}p_2^{u_2}\cdots p_k^{u_k}$, $b_0=1$, $c_1=-p_1^{w_1}\cdots p_k^{w_k}q$, $\text{gcd}(\overline{p_1},\overline{f_1})/\text{gcd}(\overline{p_1},\overline{f_1},\overline{c_1})=1$, $d_1=1$. So we write
\[
-p_1^{w_1}\cdots p_k^{w_k}q=-(p_1^{w_1-1} p_2^{w_2}\cdots p_k^{w_k}q) p_1.
\]
(In the notation of Algorithm~\ref{A6.2}, $a_1=p_1^{w_1-1} p_2^{w_2}\cdots p_k^{w_k}q$, $b_1=0$, $c_2=0$.) Eq.~\eqref{6.2} appears as 
\[
(p_1^{u_1-1}p_2^{u_2}\cdots p_k^{u_k}+ \pi  p_1^{w_1-1}p_2^{w_2}\cdots p_k^{w_k}q)p_1=f_0,
\]
so $h_0=p_1^{u_1-1}p_2^{u_2}\cdots p_k^{u_k}+ \pi p_1^{w_1-1} p_2^{w_2}\cdots p_k^{w_k}q$. It is easy to find that
\[
\begin{split}
&[I:p_1]\cap[I:\pi]\cr
=\,&\langle h_0,\pi h_1\rangle\cap \langle f_1,\pi\rangle\cr
=\,&
\begin{cases}
\langle p_1^{u_1}\cdots p_k^{u_k}+\pi p_1^{w_1}\cdots p_k^{w_k},\; \pi p_1^{v_1-1}p_2^{v_2}\cdots p_k^{v_k}\rangle&\text{if}\ u_1=v_1,\cr
\langle p_1^{u_1-1}p_2^{u_2}\cdots p_k^{u_k}+\pi p_1^{w_1-1}p_2^{w_2}\cdots p_k^{w_k},\; \pi p_1^{v_1-1}p_2^{v_2}\cdots p_k^{v_k}\rangle&\text{if}\ u_1>v_1.
\end{cases}
\end{split}
\]
The invariant sequence of $[I:p_1]\cap[I:\pi]$ is 
\[
\begin{cases}
(\overline{p_1}^{u_1}\cdots\overline{p_k}^{u_k},\; \overline{p_1}^{v_1-1} \overline{p_2}^{v_2}\cdots\overline{p_k}^{v_k})&\text{if $u_1=v_1$},\cr
(\overline{p_1}^{u_1-1}\overline{p_2}^{u_2}\cdots\overline{p_k}^{u_k},\; \overline{p_1}^{v_1-1} \overline{p_2}^{v_2}\cdots\overline{p_k}^{v_k})&\text{if $u_1>v_1$}.
\end{cases}
\]
\end{proof}

\begin{cor}\label{C7.2}
In Theorem~\ref{T7.1}, $R[x]/I$ is Frobenius if and only if for each $1\le i\le k$, $v_iw_i(u_i-v_i)=0$.
\end{cor}

\begin{proof} Immediate from Theorem~\ref{T7.1} and Algorithm~\ref{6.1}.
\end{proof} 

Next we apply Algorithm~\ref{A6.7} to Question~\ref{Q1.2} when the nilpotency $n$ of $\text{rad}\,R=\pi R$ is small. Let $I$ be an ideal of $R[x]$ with invariant sequence $(\alpha^{e_0},\dots,\alpha^{e_{n-1}})$, where $\alpha\in F[x]$ is irreducible, $e_0\ge \cdots\ge e_{n-1}\ge 0$, $e_0>0$. These are precisely the ideals such that $R[x]/I$ is artinian and local (Proposition~\ref{P6.4}). Let $r\in R[x]$ be such that $\overline r=\alpha$. Then $I$ has a canonical sequence $(f_0,\dots,f_{n-1})$ of the form 
\[
\begin{cases}
f_n=1,\cr
f_i=r^{e_i-e_{i+1}}f_{i+1}+\pi b_{i1}f_{i+2}+\cdots+\pi^{n-i-1}b_{i,n-i-1}f_n,&0\le i\le n-1,
\end{cases}
\]
where $b_{ij}\in R[x]$. (Cf. \eqref{5.1} and the proof of Proposition~\ref{P5.1}.)

For the following three theorems, we use the above notation with $n=2,3,4$, respectively. Recall from Algorithm~\ref{A6.7} that $\Lambda=\{i:0\le i\le n-1,\ e_i>e_{i+1}>0\}$, and for each $i\in\Lambda$, 
\begin{equation}\label{7.1}
K_i=\langle\pi^{i+1}f_{i+1},\dots,\pi^{n-1}f_{n-1}\rangle+\langle\pi^{i+1}rf_{i+2},\dots,\pi^{n-2}rf_{n-1},\pi^{n-1}r\rangle,
\end{equation}
and $v_i=\pi^i(f_i-r^{e_i-e_{i+1}}f_{i+1})$. The ring $R[x]/I$ is Frobenius if and only if $v_i\notin K_i$ for all $i\in\Lambda$.

\begin{thm}\label{T7.3} Assume $n=2$. So $I$ has a canonical sequence $(f_0,f_1)$ of the form
\[
\begin{cases}
f_2=1,\cr
f_1=r^{e_1},\cr
f_0=r^{e_0-e_1}f_1+\pi b_{01}f_2=r^{e_0}+\pi b_{01},
\end{cases}
\]
where $b_{01}\in R[x]$. 
\begin{itemize}
  \item [(i)] If $e_1=0$ or $e_0$, then $R[x]/I$ is Frobenius.
  \item [(ii)] If $e_0>e_1>0$, then $R[x]/I$ is Frobenius $\Leftrightarrow \overline r\nmid\overline{b_{01}}$.
\end{itemize} 
\end{thm}

\begin{proof}
(i) In this case, $\Lambda=\emptyset$, and hence $R[x]/I$ is Frobenius. 

\medskip
(ii) We have $\Lambda=\{0\}$, $v_0=\pi b_{01}$, and 
\[
K_0=\langle 0,\pi f_1\rangle+\langle 0,\pi r\rangle= \langle 0,\pi r\rangle.
\]
Thus 
\[
R[x]/I\ \text{is Frobenius}\ \Leftrightarrow\ v_0\notin K_0\ \Leftrightarrow\ \overline r\nmid\overline{b_{01}}.
\]
\end{proof}

\begin{thm}\label{T7.4} Assume $n=3$. So $I$ has a canonical sequence $(f_0,f_1,f_2)$ of the form
\[
\begin{cases}
f_3=1,\vspace{1mm}\cr
f_2=r^{e_2},\vspace{1mm}\cr
f_1=r^{e_1-e_2}f_2+\pi b_{11}f_3\cr
\phantom{f_1}=r^{e_1}+\pi b_{11},\vspace{1mm}\cr
f_0=r^{e_0-e_1}f_1+\pi b_{01}f_2+\pi^2b_{02}f_3\cr
\phantom{f_0}=r^{e_0}+\pi(b_{11}r^{e_0-e_1}+b_{01}r^{e_2})+\pi^2b_{02},
\end{cases}
\]
where $b_{ij}\in R[x]$. 
\begin{itemize}
  \item [(i)] If $|\{e_0,e_1,e_2,0\}|=2$, then $R[x]/I$ is Frobenius.
  \item [(ii)] If $e_0=e_1>e_2>0$, then $R[x]/I$ is Frobenius $\Leftrightarrow \overline r\nmid\overline{b_{11}}$.
  \item [(iii)] If $e_0>e_1=e_2>0$, then $R[x]/I$ is Frobenius $\Leftrightarrow \overline r\nmid\overline{b_{02}}-\overline{b_{01}}\,\overline{b_{11}}$.
  \item [(iv)] If $e_0>e_1>e_2=0$, then $R[x]/I$ is Frobenius $\Leftrightarrow \overline r\nmid\overline{b_{01}}$. 
  \item [(iv)] If $e_0>e_1>e_2>0$, then $R[x]/I$ is Frobenius $\Leftrightarrow \overline r\nmid\overline{b_{11}},\,\overline{b_{01}}$. 
\end{itemize} 
\end{thm}

\begin{proof}
The proof of (i) and (ii) is identical to the proof of (i) and (ii) in Theorem~\ref{T7.3}. For the proof of the remaining cases, we state the results for $v_i$ and $K_i$, $i\in\Lambda$, but omit the routine computations that lead to these results. 

\medskip
(iii) We have $\Lambda=\{0\}$, $v_0=\pi b_{01}r^{e_2}+\pi^2 b_{02}$, and $K_0=\langle\pi r^{e_1}+\pi^2 b_{11},\pi^2 r\rangle$. It is easy to see that $v_0\notin K_0\Leftrightarrow r\nmid\overline{b_{02}}-\overline{b_{01}}\,\overline{b_{11}}$.

\medskip
(iv) We have $\Lambda=\{0\}$, $v_0=\pi b_{01}+\pi^2 b_{02}$, and $K_0=\langle\pi r,\pi^2 \rangle$. Clearly, $ v_0\notin K_0\Leftrightarrow r\nmid\overline{b_{01}}$.

\medskip
(v) We have $\Lambda=\{0,1\}$, $ v_1=\pi^2 b_{11}$, $K_1=\langle\pi^2 r\rangle$. So $v_1\notin K_1\Leftrightarrow r\nmid\overline{b_{11}}$.

Assume $r\nmid\overline{b_{11}}$. We have $K_0=\langle\pi r^{e_2+1},\pi^2\rangle$ and $v_0=\pi b_{01}r^{e_2}+\pi^2 b_{02}$. Clearly, $v_0\notin K_0\Leftrightarrow r\nmid\overline{b_{01}}$.
\end{proof} 

\begin{thm}\label{T7.5}
Assume $n=4$. So $I$ has a canonical sequence $(f_0,f_1,f_2,f_3)$ of the form
\[
\begin{cases}
f_4=1,\cr
f_3=r^{e_3},\vspace{1mm}\cr
f_2=r^{e_2-e_3}f_3+\pi b_{21}f_4\cr
\phantom{f_2}=r^{e_2}+\pi b_{21},\vspace{1mm}\cr
f_1=r^{e_1-e_2}f_2+\pi b_{11}f_3+\pi^2b_{12}f_4\cr
\phantom{f_2}=r^{e_1}+\pi(b_{21}r^{e_1-e_2}+b_{11}r^{e_3})+\pi^2b_{12},\vspace{1mm}\cr
f_0=r^{e_0-e_1}f_1+\pi b_{01}f_2+\pi^2b_{02}f_3+\pi^2 b_{03}f_4\cr
\phantom{f_2}=r^{e_0}+\pi(b_{21}r^{e_0-e_2}+b_{11}r^{e_0-e_1+e_3}+b_{01}r^{e_2})\cr
\phantom{f_2=}+\pi^2(b_{12}r^{e_0-e_1}+b_{01}b_{21}+b_{02}r^{e_3})+\pi^3b_{03},
\end{cases}
\]
where $b_{ij}\in R[x]$. 
\begin{itemize}
  \item [(i)] If $|\{e_0,e_1,e_2,e_3,0\}|=2$, then $R[x]/I$ is Frobenius.
  \item [(ii)] If $e_0=e_1=e_2>e_3>0$, then $R[x]/I$ is Frobenius $\Leftrightarrow \overline r\nmid\overline{b_{21}}$.
  \item [(iii)] If $e_0=e_1>e_2=e_3>0$, then $R[x]/I$ is Frobenius $\Leftrightarrow \overline r\nmid\overline{b_{12}}-\overline{b_{11}}\,\overline{b_{21}}$.
  \item [(iv)] If $e_0=e_1>e_2>e_3=0$, then $R[x]/I$ is Frobenius $\Leftrightarrow \overline r\nmid\overline{b_{11}}$. 
  \item [(v)] If $e_0=e_1>e_2>e_3>0$, then $R[x]/I$ is Frobenius $\Leftrightarrow \overline r\nmid\overline{b_{21}},\overline{b_{11}}$. 
  \item [(vi)] If $e_0>e_1=e_2=e_3>0$, then $R[x]/I$ is Frobenius $\Leftrightarrow \overline r\nmid\overline{b_{21}}(\overline{b_{02}}-\overline{b_{01}}\,\overline{b_{11}})-(\overline{b_{03}}-\overline{b_{01}}\,\overline{b_{12}})$. 
  \item [(vii)] If $e_0>e_1=e_2>e_3=0$, then $R[x]/I$ is Frobenius $\Leftrightarrow \overline r\nmid\overline{b_{02}}-\overline{b_{01}}\,\overline{b_{11}}$.
  \item [(viii)] If $e_0>e_1>e_2=e_3=0$, then $R[x]/I$ is Frobenius $\Leftrightarrow \overline r\nmid\overline{b_{01}}$.
  \item [(ix)] If $e_0>e_1=e_2>e_3>0$, then $R[x]/I$ is Frobenius $\Leftrightarrow \overline r\nmid\overline{b_{21}}$ and $\overline r\nmid \overline{b_{02}}-\overline{b_{01}}\,\overline{b_{11}}$.
  \item [(x)] If $e_0>e_1>e_2=e_3>0$, then $R[x]/I$ is Frobenius $\Leftrightarrow \overline r\nmid\overline{b_{12}}-\overline{b_{11}}\,\overline{b_{21}}$ and $\overline r\nmid \overline{b_{01}}$.
  \item [(xi)] If $e_0>e_1>e_2>e_3=0$, then $R[x]/I$ is Frobenius $\Leftrightarrow \overline r\nmid\overline{b_{11}},\,\overline{b_{01}}$.
  \item [(xii)] If $e_0>e_1>e_2>e_3>0$, then $R[x]/I$ is Frobenius $\Leftrightarrow \overline r\nmid\overline{b_{21}},\,\overline{b_{11}},\,\overline{b_{01}}$.
\end{itemize} 
\end{thm} 

\begin{proof}
The proof of (i) -- (v) is identical to the proof (i) - (v) in Theorem~\ref{T7.4}. For the proof of the remaining cases, again we omit the routine and tedious computations.

\medskip
(vi) We have $\Lambda=\{0\}$, $K_0=\langle\pi f_1,\pi^2f_2,\pi^3r\rangle$, $v_0=\pi b_{01}f_2+\pi^2 b_{02}f_3+\pi^3b_{03}$. It follows that $v_0\notin K_0 \Leftrightarrow\overline r\nmid\overline{b_{21}}(\overline{b_{02}}-\overline{b_{01}}\,\overline{b_{11}})-(\overline{b_{03}}-\overline{b_{01}}\,\overline{b_{12}})$.

\medskip
(vii) We have $\Lambda=\{0\}$, $K_0=\langle\pi f_1,\pi^2r,\pi^3\rangle$, $v_0=\pi b_{01}f_2+\pi^2 b_{02}f_3+\pi^3b_{03}$. It follows that $v_0\notin K_0 \Leftrightarrow\overline r\nmid\overline{b_{02}}-\overline{b_{01}}\,\overline{b_{11}}$.

\medskip
(viii) We have $\Lambda=\{0\}$, $K_0=\langle\pi r,\pi^2,\pi^3\rangle$, $v_0=\pi b_{01}f_2+\pi^2 b_{02}f_3+\pi^3b_{03}$. It is clear that $v_0\notin K_0 \Leftrightarrow\overline r\nmid\overline{b_{01}}$.

\medskip
(ix) We have $\Lambda=\{0,2\}$, $K_2=\langle\pi^3r\rangle$, $v_2=\pi^3 b_{21}$. Clearly, $v_2\notin K_2 \Leftrightarrow\overline r\nmid\overline{b_{21}}$.

Assume $\overline r\nmid\overline{b_{21}}$. We have $K_0=\langle\pi f_1,\pi^2rf_3,\pi^3\rangle$, $v_0=\pi b_{01}f_2+\pi^2 b_{02}f_3+\pi^3b_{03}$. It follows that $v_0\notin K_0 \Leftrightarrow\overline r\nmid\overline{b_{02}}-\overline{b_{01}}\,\overline{b_{11}}$.

\medskip
(x) We have $\Lambda=\{0,1\}$, $K_1=\langle\pi^2f_2,\pi^3r\rangle$, $v_1=\pi^2 b_{11}f_3+\pi^3b_{12}$. It follows that $v_1\notin K_1 \Leftrightarrow\overline r\nmid\overline{b_{12}}-\overline{b_{11}}\,\overline{b_{21}}$.

Assume $\overline r\nmid\overline{b_{12}}-\overline{b_{11}}\,\overline{b_{21}}$. We have $K_0=\langle\pi rf_2,\pi^2f_3,\pi^3\rangle$, $v_0=\pi b_{01}f_2+\pi^2 b_{02}f_3+\pi^3b_{03}$. It follows that $v_0\notin K_0 \Leftrightarrow\overline r\nmid\overline{b_{01}}$.

\medskip
(xi) We have $\Lambda=\{0,1\}$, $K_1=\langle\pi^2r,\pi^3\rangle$, $v_1=\pi^2 b_{11}+\pi^3b_{12}$. So $v_1\notin K_1 \Leftrightarrow\overline r\nmid\overline{b_{11}}$.

Assume $\overline r\nmid\overline{b_{11}}$. We have $K_0=\langle\pi r^{e_2+1},\pi^2,\pi^3\rangle$, $v_0=\pi b_{01}f_2+\pi^2 b_{02}+\pi^3b_{03}$. Clearly, $v_0\notin K_0 \Leftrightarrow\overline r\nmid\overline{b_{01}}$.

\medskip
(xii) We have $\Lambda=\{0,1,2\}$, $K_2=\langle\pi^3r\rangle$, $v_2=\pi^3 b_{21}$. So $v_2\notin K_2 \Leftrightarrow\overline r\nmid\overline{b_{21}}$.

Assume $\overline r\nmid\overline{b_{21}}$. We have $K_1=\langle\pi^2 r^{e_3+1},\pi^3\rangle$, $v_1=\pi^2 b_{11}r^{e_3}+\pi^3 b_{12}$. So $v_1\notin K_1 \Leftrightarrow\overline r\nmid\overline{b_{11}}$.

Now assume $\overline r\nmid\overline{b_{21}}$ and $\overline r\nmid\overline{b_{11}}$. We have $K_0=\langle\pi rf_2,\pi^2r^{e_3},\pi^3\rangle$, $v_0=\pi b_{01}f_2+\pi^2 b_{02}f_3+\pi^3b_{03}$. Clearly, $v_0\notin K_0 \Leftrightarrow\overline r\nmid\overline{b_{01}}$.
\end{proof}
 
It we continue in this manner to find the necessary and sufficient conditions for $R[x]/I$ to be Frobenius and local with $n=5,6,\dots$, the computations involved are expected to be more complicated. But aside from that, there are no other inhibiting obstacles. 


\end{document}